\documentclass{amsart}

\usepackage{amsfonts}
\usepackage{amssymb}
\usepackage{mathtools}
\usepackage{tikz}
\usetikzlibrary{arrows,decorations.pathreplacing}
\usetikzlibrary{shapes}%,snakes}
\usetikzlibrary{arrows,calc}

\usepackage[
	pdftitle={},
	pdfauthor={},
	ocgcolorlinks,
	linkcolor=linkblue,
	citecolor=linkred,
	urlcolor=linkred]
{hyperref}
\usepackage{color}
\definecolor{linkred}{rgb}{0.75,0,0}
\definecolor{linkblue}{rgb}{0,0,0.75}
\usepackage{multicol}
\usepackage{multirow}
\usepackage{enumerate}

\theoremstyle{plain}
\newtheorem{theorem}{Theorem}

\newtheorem{proposition}{Proposition}[section]
\newtheorem{lemma}[proposition]{Lemma}

\newcommand{\bt}{\begin{theorem}}
\newcommand{\et}{\end{theorem}}

\theoremstyle{definition}
\newtheorem{definition}[proposition]{Definition}

\newcommand{\beq}{\begin{equation}}
\newcommand{\eeq}{\end{equation}}

\newcommand{\cal}{\mathcal}
\newcommand{\Res}{\mathop{\,\rm Res\,}}

\newcommand{\cc}{\mathcal{C}}
\newcommand{\cd}{\mathcal{D}}

\newcommand{\cl}{\mathcal{L}}
\newcommand{\ce}{\mathcal{E}}
\newcommand\cn{\mathcal{N}}

\newcommand{\bc}{\mathbb{C}}

\newcommand{\bk}{\mathbb{K}}

\newcommand{\bn}{\mathbb{N}}

\newcommand{\bq}{\mathbb{Q}}
\newcommand{\br}{\mathbb{R}}
\newcommand{\bz}{\mathbb{Z}}

\newcommand{\WP}{^{WP}\hspace{-.4mm}}

\newcommand{\modm}{\cal M}

\newcommand{\pd}[2]{\frac{\partial#1}{\partial#2}}

\begin{document}
	
\title[Super volumes and KdV tau functions]{Super volumes and KdV tau functions}
\author{Alexander Alexandrov}
\address{Center for Geometry and Physics, Institute for Basic Science (IBS), Pohang 37673, Korea}
\email{\href{mailto:alex@ibs.re.kr}{alex at ibs.re.kr}}
\author{Paul Norbury}
\address{School of Mathematics and Statistics, University of Melbourne, VIC 3010, Australia}
\email{\href{mailto:norbury@unimelb.edu.au}{norbury at unimelb.edu.au}}
\thanks{}
\date{\today}

\begin{abstract} 
Weil-Petersson volumes of the moduli space of curves are deeply related to the Kontsevich-Witten KdV tau function.  They possess a Virasoro symmetry which comes out of recursion relations between the volumes due to Mirzakhani.  Similarly, the super Weil-Petersson volumes of the moduli space of super curves with Neveu-Schwarz punctures are related to the Br\'ezin-Gross-Witten (BGW) tau function of the KdV hierarchy and satisfy a recursion due to Stanford and Witten, analogous to Mirzakhani's recursion.  In this paper we prove that by also allowing Ramond punctures, the super Weil-Petersson volumes are related to the {\em generalised} BGW KdV tau function, which is a one parameter deformation of the BGW tau function.  This allows us to prove that these new super volumes also satisfy the Stanford-Witten recursion.
%  We extend this result here, to also allow Ramond punctures, proving that the The realtionship with  The BGW tau function was known to be related to the moduli space of curves via intersection numbers of the $\Theta$-class.  The construction of the $\Theta$-class uses the moduli space of stable spin curves with Neveu-Schwarz punctures. 
  \end{abstract}

\maketitle

\tableofcontents

\section{Introduction}

The Weil-Petersson volumes of the moduli spaces of genus $g$ curves with $n$ marked points $\modm_{g,n}$ play an important role in both pure mathematics and physics.  They deliver reciprocal benefits to both mathematics and physics by creating connections between the two.  Their relationship to two-dimensional quantum gravity leads to integrability of partition functions constructed from the volumes.  This has deep consequences in algebraic geometry, where the volumes arise as intersection numbers on the moduli spaces  $\overline{\modm}_{g,n}$ of genus $g$ stable curves $(C,p_1,\ldots,p_n)$ with $n$ labeled smooth points, for any $g,n\in\bn=\{0,1,2,...\}$ satisfying $2g-2+n>0$, following work of Zograf \cite{ZogWei}.
Recently, super Weil-Petersson volumes of the moduli spaces of genus $g$ super curves with $n$ marked points $\widehat{\modm}_{g,n}$ have been shown to follow a similar path.  They satisfy a Virasoro symmetry, realised via a recursion between volumes analogous to a recursion between classical volumes due to Mirzakhani  \cite{MirSim}.  The super volume recursions were explained heuristically via super geometry techniques in \cite{SWiJTG} and proven rigorously via algebro-geometric techniques in \cite{NorEnu}, where a relationship with the BGW tau function was also proven.

The algebro-geometric treatment of super Weil-Petersson volumes still takes place on $\overline{\modm}_{g,n}$, rather than over moduli spaces of stable  super curves.  The idea is to consider all possible spin structures on a stable curve, which is achieved via the push-forward of cohomology classes from the moduli space of stable spin curves, $\overline{\modm}_{g,n}^{\rm spin}$, and this serves as the reduced space of the moduli spaces of stable super curves.  A spin structure has two types of behaviour at marked points and nodes known as {\em Neveu-Schwarz} and {\em Ramond}---see Definition~\ref{RNS}.  The results in \cite{NorEnu}, in particular the relationship with the BGW tau function, assume Neveu-Schwarz behaviour at all marked points.

In this paper we show that the generalised BGW tau function arises, when we also take into account Ramond marked points.  We first prove that the generalised BGW tau function gives a generating function for intersection numbers of so-called {\em spin classes} over $\overline{\modm}_{g,n}$ given by the push-forward of cohomology classes from $\overline{\modm}_{g,n}^{\rm spin}$ using Neveu-Schwarz and Ramond marked points.  This gives an analogue of the Kontsevich-Witten theorem, described below, which gives a generating function for intersection numbers of the trivial class.

A tau function $Z(t_0,t_1,...)$ of the Korteweg-de Vries (KdV) hierarchy (equivalently the Kadomtsev-Petviashvili hierarchy with solutions dependent only on  odd times $p_{2k+1}=t_k/(2k+1)!!$) produces a solution $U=\hbar\frac{\partial^2}{\partial t_0^2}\log Z$ of a hierarchy of partial differential equations.  The first equation in the hierarchy is the KdV equation:
\[
U_{t_1}=UU_{t_0}+\frac{\hbar}{12}U_{t_0t_0t_0},\quad U(t_0,0,0,...)=f(t_0)
\]
and later equations $U_{t_k}=P_k(U,U_{t_0},U_{t_0t_0},...)$ for $k>1$ determine $U$ uniquely from $U(t_0,0,0,...)$.  

Since the seminal work of Witten \cite{WitTwo}, deep relationships have been uncovered between the KdV equation and intersection numbers on the moduli space $\overline{\modm}_{g,n}$ %(Deligne-Mumford stack)
of genus $g$ stable curves $(C,p_1,\ldots,p_n)$ with $n$ labeled smooth points, for any $g,n\in\bn=\{0,1,2,...\}$ satisfying $2g-2+n>0$.

The first such relationship is given by the Kontsevich-Witten KdV tau function which is a generating function for the following tautological intersection numbers over $\overline{\modm}_{g,n}$.  Define the $\psi$ class $\psi_i=c_1(L_i)\in H^{2}(\overline{\modm}_{g,n},\bq)$ for each $i=1,...,n$ to be the first Chern class of the line bundle $L_i\to\overline{\modm}_{g,n}$ which has fibre at the point $(C,p_1,\ldots,p_n)$ given by the cotangent space $T^*_{p_i}C$.  With respect to the forgetful map $\pi:\overline{\modm}_{g,n+1}\to\overline{\modm}_{g,n}$ which forgets the point $p_{n+1}$, the line bundle $L_i= p_i^*\omega_\pi$ is the pullback of the relative dualising sheaf $\omega_\pi$ by the section $p_i:\overline{\modm}_{g,n}\to\overline{\modm}_{g,n+1}$.  The Kontsevich-Witten KdV tau function \cite{KonInt,WitTwo} is defined by
\begin{equation}   \label{tauKW}
Z^{\text{KW}}(\hbar,t_0,t_1,...)=\exp\sum_{g,n}\frac{\hbar^{g-1}}{n!}\sum_{\vec{k}\in\bn^n}\int_{\overline{\modm}_{g,n}}\prod_{i=1}^n\psi_i^{k_i}t_{k_i}.
\end{equation}
The partial differential equations of the KdV hierarchy satisfied by $Z^{\text{KW}}(\hbar,t_0,t_1,...)$ define recursion relations between the intersection numbers of $\psi$ classes, which determines them uniquely from the initial calculation $\int_{\overline{\modm}_{0,3}}1=1$.

  %hierarchy.%$$ F^{\text{KW}}(\hbar,t_0,t_1,...)=\sum_g\hbar^{2g}\langle\tau_0^{k_0}\tau_1^{k_1}...\rangle_g\frac{t_0^{k_0}}{k_0!}\frac{t_1^{k_1}}{k_1!}...$$

The generalised Br\'ezin-Gross-Witten function $Z^{BGW}(s,\hbar,t_0,t_1,\dots)$, introduced in \cite{MMSUni} and investigated in \cite{AleCut}, is the unique normalised ($Z^{BGW}(s,\hbar,\{t_k=0\})=1$) tau function of the KdV hierarchy that satisfies the following homogeneity condition:
\[
\left(\frac{\partial}{\partial t_0}-\sum_{k=0}^\infty (2k+1)t_k\frac{\partial}{\partial t_k}
\right)\log Z^{BGW}%(s,\hbar,t_0,t_1,\dots)
=\frac12\hbar^{-1} s^2+\frac18.
\]
It is a homogeneity condition since it implies that $F^{BGW}_g$, the genus $g$ free energy (i.e. coefficient of $\hbar^{g-1}$ in $\log Z^{BGW}$), for $g>1$ is a series in $T_k=t_k/(1-t_0)^{2k+1}$ for $k=1,2,...$.  %It is proven in \cite{NorSup} that  $Z^\Omega(\hbar,s,t_0,t_1,...)$ satisfies the same homogeneity condition.

When $s=0$, $Z^{BGW}$ specialises to
the Br\'ezin-Gross-Witten tau function that arises out of a unitary matrix model studied in \cite{BGrExt,GWiPos}.  %It is defined by the initial condition \begin{equation*} \hbar\frac{\partial^2}{\partial t_0^2}\log Z^{\text{BGW}}(t_0,0,0,...)=U^{\text{BGW}}(t_0,0,0,...)=\frac{\hbar}{8(1-t_0)^2}. \end{equation*}
It was proven in \cite{CGGRel} that the Br\'ezin-Gross-Witten tau function gives a generating function for intersection numbers over $\overline{\modm}_{g,n}$ analogous to $Z^{\text{KW}}$:
\begin{equation}  \label{BGWtheta}
Z^{BGW}(s=0,\hbar,t_0,t_1,...)=\exp\sum_{g,n}\frac{\hbar^{g-1}}{n!}\sum_{\vec{k}\in\bn^n}\int_{\overline{\modm}_{g,n}}\Theta_{g,n}\cdot\prod_{i=1}^n\psi_i^{k_i}t_{k_i}
\end{equation}
for $\Theta_{g,n}\in H^{4g-4+2n}(\overline{\modm}_{g,n},\bq)$ defined in \cite{NorNew}.  Again, the KdV hierarchy produces recursion relations between the intersection numbers $\int_{\overline{\modm}_{g,n}}\Theta_{g,n}\cdot\prod_{j=1}^n\psi_j^{k_j}$, which determines them uniquely from the initial calculation $\int_{\overline{\modm}_{1,1}}\Theta_{g,n}=\frac18$.

In this paper we show that the generalised Br\'ezin-Gross-Witten tau function is related to intersection numbers of natural cohomology classes over $\overline{\modm}_{g,n}$ analogously to \eqref{tauKW} and \eqref{BGWtheta}. 
We define a natural collection of cohomology classes, denoted spin classes, via a geometric construction over the moduli space of stable spin curves $\overline{\modm}_{g,n}^{\rm spin}$ which can be used to reconstruct  $Z^{BGW}(s,\hbar,t_0,t_1,\dots)$.  The  behaviour of the spin structure at each labeled point is Ramond or Neveu-Schwarz, indicated by $\sigma\in\{0,1\}^n$ which assigns 0 or 1 to each of the $n$ labeled points, and determines components $\overline{\modm}_{g,\sigma}^{\rm spin}\subset\overline{\modm}_{g,n}^{\rm spin}$.   There is a natural vector bundle $E_{g,n}\to\overline{\modm}_{g,n}^{\rm spin}$, and the top Chern class on each component pushes forward to give the spin class, up to scale, denoted $\Omega_{g,n}^\sigma\in H^*(\overline{\modm}_{g,n},\bq)$.    See Section~\ref{spincohft} for precise definitions.
Define
\begin{equation}  \label{tauomega}
Z^\Omega(s,\hbar,t_0,t_1,...)=\exp\sum_{g,n}\frac{\hbar^{g-1}}{n!}\sum_{\vec{k}\in\bn^n}\sum_{m=0}^\infty\frac{s^m}{m!}\int_{\overline{\modm}_{g,n+m}}\hspace{-5mm}\Omega_{g,n+m}^{(1^n,0^m)}.\prod_{i=1}^n\psi_i^{k_i}t_{k_i}.
\end{equation}
We see that the powers of $s$ in \eqref{tauomega} keep track of the number of Ramond points.  Note that the sum over $(g,n)$ in \eqref{tauomega} also includes contributions from the unstable range, $(0,1)$, $(0,2)$.%, as well as $t$-independent terms with $n=0$.  

The following theorem gives a relationship of intersection numbers of the spin classes over $\overline{\modm}_{g,n}$ with the generalised Br\'ezin-Gross-Witten tau function.
\begin{theorem} \label{BGW=spin} 
\[Z^{BGW}(\hbar,s,t_0,t_1,...)=Z^\Omega(\hbar,s,t_0,t_1,...).\]
\end{theorem}
The classes $\Theta_{g,n}$ in \eqref{BGWtheta} are a special case of the spin classes, where all marked points are of Neveu-Schwarz type: 
\[\Theta_{g,n}:=\Omega_{g,n}^{(1^n)}.\] 
Hence \eqref{BGWtheta} gives $Z^{BGW}|_{s=0}=Z^\Omega|_{s=0}$ and Theorem~\ref{BGW=spin} extends \eqref{BGWtheta} to all terms of the series expansion in $s$.

The intersection numbers in \eqref{BGWtheta} were discovered to be related to volumes of super Riemann surfaces with Neveu-Schwarz marked points by Stanford and Witten \cite{SWiJTG}.   Likewise, the spin class intersection numbers are related to volumes of super Riemann surfaces with Neveu-Schwarz and Ramond marked points, and a consequence of Theorem~\ref{BGW=spin} is a relation of the generalised Br\'ezin-Gross-Witten tau function to  super Riemann surfaces.  The main application of this relation here is a proof of a recursion satisfied by the super volumes with Neveu-Schwarz and Ramond marked points. 

The moduli space of stable, genus $g$ super Riemann surfaces with $n$ marked points $\widehat{\modm}_{g,n}$ defines a superstack with reduced space the moduli space $\overline{\modm}_{g,n}^{\rm spin}$ of stable genus $g$ spin curves with $n$ marked points.  As a smooth supermanifold (or superorbifold) $\widehat{\modm}_{g,n}$ is represented by the sheaf of smooth sections of %the exterior algebra of the bundle 
$\Lambda^*E_{g,n}$, where $E_{g,n}\to\overline{\modm}_{g,n}^{\rm spin}$ is the vector bundle defined in Definition~\ref{obsbun}, and used to define the classes $\Omega_{g,n}^\sigma$.  The restriction of $E_{g,n}$ to a component $\overline{\modm}_{g,\sigma}^{\rm spin}\subset\overline{\modm}_{g,n}^{\rm spin}$ has top Chern class which is used together with the Weil-Petersson form, represented by the class $2\pi^2\kappa_1$, to define the super volume of $\overline{\modm}_{g,\sigma}^{\rm spin}$ as follows.  For $\sigma=(1^n,0^m)$, via the pushforward $\overline{\modm}_{g,\sigma}^{\rm spin}\to\overline{\modm}_{g,n+m}$ the super volume of $\overline{\modm}_{g,\sigma}^{\rm spin}$ is given by $\int_{\overline{\modm}_{g,n+m}}\Omega_{g,n+m}^{(1^n,0^m)}\exp(2\pi^2\kappa_1)$.   In the following series in the variable $s$, we further insert $\psi$ classes at Neveu-Schwarz points and sum over all such volumes with a given number of Neveu-Schwarz points
\begin{equation}  \label{vgns}
\widehat{V}_{g,n}\WP(s,L_1,...,L_n):=\sum_{m=0}^\infty\frac{s^m}{m!}\int_{\overline{\modm}_{g,n+m}}\hspace{-4mm}\Omega_{g,n+m}^{(1^n,0^m)}\exp\left(2\pi^2\kappa_1+\frac12\sum_{i=1}^nL_i^2\psi_i\right).
\end{equation}
Note that the coefficients of this series in $s$ is given by polynomials in the variables $L_1,...,L_n$, or in other words $\widehat{V}_{g,n}\WP(s,L_1,...,L_n)\in\br[L_1,...,L_n][[s^2]]$.   If we remove the $2\pi^2\kappa_1$ terms from the integrand in \eqref{vgns}, or equivalently take the terms of homogeneous degree $2g-2+m$ in the coefficient of $s^m$, then this reduces to spin class intersection numbers, and by summing over $g$ and $n$ it produces $\log Z^{\Omega}(\hbar,s,\{t_k\})$ via a change of coordinate $L_i^{2k}=2^kk!t_k$.
%\[Z^{\Omega}(\hbar,s,\{t_k\})=\exp\sum_{g,n}\frac{\hbar^{g-1}}{n!}\widehat{V}_{g,n}\WP(s,L_1,...,L_n)|_{\{L_i^{2k}=2^kk!t_k\}}.\]

Define
\[ D(x,y,z)=\frac{\sinh\frac{x}{4}\sinh\frac{y+z}{4}}{\cosh\frac{x-y-z}{4}\cosh\frac{x+y+z}{4}}
\]
and $R(x,y,z)=\tfrac12(D(x+y,z,0)+D(x-y,z,0))$.   The relation in Theorem~\ref{BGW=spin} can be used to prove the following theorem which was conjectured in \cite{NorSup} and proven in some cases there.
\begin{theorem}  \label{thmsup}
The volumes of the moduli spaces of super Riemann surfaces satisfy the Stanford-Witten recursion:
\begin{align} \label{reconj}
L_1&\widehat{V}_{g,n}\WP(s,\vec{L})=\tfrac12\int_0^\infty\hspace{-2mm}\int_0^\infty\hspace{-2mm} xyD(L_1,x,y)P_{g,n+1}(x,y,L_K)dxdy\\
&+\sum_{j=2}^n\int_0^\infty \hspace{-2mm}xR(L_1,L_j,x)\widehat{V}_{g,n-1}\WP(s,x,L_{K\backslash\{j\}})dx+\delta_{1,n}(\tfrac{s^2\delta_{0,g}}{2}+\tfrac{\delta_{1,g}}{8})L_1
\nonumber 
\end{align}
for $K=(2,...,n)$ and
\[P_{g,n+1}(x,y,L_K)=\widehat{V}_{g-1,n+1}\WP(s,x,y,L_K)+\hspace{-3mm}\mathop{\sum_{g_1+g_2=g}}_{I \sqcup J = K}\hspace{-2mm}\widehat{V}_{g_1,|I|+1}\WP(s,x,L_I)\widehat{V}_{g_2,|J|+1}\WP(s,y,L_J).
\]
\end{theorem}
The recursion \eqref{reconj} resembles Mirzakhani's recursion \cite{MirSim} for Weil-Petersson volumes of the classical moduli space.  Mirzakhani's argument was adapted in a heuristic argument by Stanford and Witten \cite{SWiJTG} which was proven using algebraic geometry in \cite{NorEnu}  for the $s=0$  case.  The kernels $D(x,y,z)$ and $R(x,y,z)$ are independent of $s$ and Theorem~\ref{thmsup} shows that essentially the same recursion is satisfied by the $s$ deformation of the volumes.  The only $s$-dependent term gives the initial disk condition.

The proof of Theorem~\ref{BGW=spin}  spans most of sections 2 and 3.  In Section~\ref{sec:inter} we define spin intersection numbers and prove a useful relationship with intersection numbers of polynomials in kappa classes.  The  intersection numbers of kappa classes are used to produce a KdV tau function which provides an integrable structure used in Section~\ref{sec:integ} to complete the proof of Theorem~\ref{BGW=spin} . Theorem~\ref{thmsup} is equivalent to Virasoro constraints, and we prove these indirectly via identifying KdV tau functions, one of which is known to satisfy  Virasoro constraints.  In the final section we describe the relationship of the results here with topological recursion.

{\bf Acknowledgements} A.~A. was supported by the Institute for Basic Science (IBS-R003-D1). A.~A. is grateful to MATRIX for hospitality. 

\section{Intersection numbers on $\overline{\modm}_{g,n}$}  \label{sec:inter}

In this section we define the spin classes $\Omega_{g,n+m}^{(1^n,0^m)}\in H^*( \overline{\cal M}_{g,n},\bq)$ which use the extra data of a spin structure on each curve.  We also define cohomology classes $\bk\in H^*( \overline{\cal M}_{g,n},\bq)$ built from kappa classes which will provide an intermediary step towards the proof of Theorem~\ref{BGW=spin}.

\subsection{Spin classes}   \label{spincohft}

A spin structure on a curve $C$ is defined to be a square root of its canonical bundle $\omega_C$. On a curve with marked points $(C,D)$ for $D=\{p_1,...,p_n\}$ one instead considers the log-canonical bundle $\omega_C^{\text{log}}=\omega_C(D)$ and possibly twists it further at marked points so that a square root is defined.  It will be convenient to instead work over stable twisted curves, following \cite{AJaMod}.  A twisted curve, with group $\bz_2$, is a 1-dimensional orbifold, or stack, $\cc$ such that generic points of $\cc$ have trivial isotropy group, while labeled points $p_i\in\cc$ and nodes have non-trivial isotropy group $\bz_2$.

Define the moduli space of stable twisted spin curves by
$$\overline{\modm}_{g,n}^{\rm spin}=\{(\cc,\theta,p_1,...,p_n,\phi)\mid \phi:\theta^2\stackrel{\cong}{\longrightarrow}\omega_{\cc}^{\text{log}}\}/\sim.
$$
The pair $(\theta,\phi)$ is a {\em spin structure} on $\cc$ where $\omega_{\cc}^{\text{log}}$ and $\theta$ are line bundles over the stable twisted curve $\cc$ with labeled orbifold points $p_j$ and $\deg\theta=g-1+\frac12n$.     A line bundle $L$ over a twisted curve $\cc$ is a locally equivariant bundle over the local charts, such that at each nodal point there is an equivariant isomorphism of fibres.  Hence each orbifold point $p$ associates a representation of $\bz_2$ on $L|_p$ acting by multiplication by $\exp(\pi i\sigma_p)$ for $\sigma_p\in\{0,1\}$.  
\begin{definition}  \label{RNS}
An orbifold point $p$ of a spin structure is defined to be of type {\em Ramond} if $\sigma_p=0$ and of type {\em Neveu-Schwarz} if $\sigma_p=1$.
\end{definition}
The equivariant isomorphism of fibres over nodal points forces a balanced condition at each node $p$ given by $\sigma_{p_+}=\sigma_{p_-}$ for $p_\pm$ corresponding to the node on each irreducible component.

Define a vector bundle over $\overline{\modm}_{g,n}^{\rm spin}$ using the dual bundle $\theta^{\vee}$ on each stable twisted curve as follows.  Denote by $\ce$ the universal spin structure on the universal stable twisted spin curve over $\overline{\modm}_{g,n}^{\rm spin}$.  Given a map $S\to\overline{\modm}_{g,n}^{\rm spin}$, $\ce$ pulls back to $\theta$ giving a family $(\cc,\theta,p_1,...,p_n,\phi)$ where $\pi:\cc\to S$ has stable twisted curve fibres, $p_i:S\to\cc$ are sections with orbifold isotropy $\bz_2$ and $\phi:\theta^2\stackrel{\cong}{\longrightarrow}\omega_{\cc/S}^{\text{log}}=\omega_{\cc/S}(p_1,..,p_n)$.  The derived push-forward sheaf of $\ce^{\vee}$ over $\overline{\modm}_{g,n}^{\rm spin}$ gives rise to a bundle as follows.  The space of global sections of $\pi_*\ce^{\vee}$ vanishes, i.e. $R^0\pi_*\ce^{\vee}=0$, since
\[\deg\theta^{\vee}=1-g-\frac12n<0
\]
and furthermore, for any irreducible component $\cc'\stackrel{i}{\to}\cc$,
\[\deg\theta^{\vee}|_{\cc'}<0
\]
since  $i^*\omega_{\cc/S}^{\text{log}}=\omega_{\cc'/S}^{\text{log}}$ and $\cc'$ is stable  guarantees its log canonical bundle has negative degree.  Thus $R^1\pi_*\ce^\vee$ defines the following bundle.
\begin{definition}  \label{obsbun}
Define a bundle $E_{g,n}=-R\pi_*\ce^\vee$ over $\overline{\modm}_{g,n}^{\rm spin}$. 
\end{definition}
\noindent The bundle $E_{g,n}$ has fibre $H^1(\cc,\theta^{\vee})$ over the curve $(\cc,D)$ with spin structure $\theta$.

The moduli space of stable twisted spin curves decomposes into components labeled by $\sigma\in\{0,1\}^n$ satisfying $\displaystyle|\sigma|+n\in 2\bz$:
$$\overline{\modm}_{g,n}^{\rm spin}=\bigsqcup_{\sigma\in\{0,1\}^n}\overline{\modm}_{g,\sigma}^{\rm spin}
$$
where $\overline{\modm}_{g,\sigma}^{\rm spin}$ consists of spin curves with induced representation at labeled points determined by $\sigma$ as follows.
 The representations induced by the orbifold line bundle $\theta$ at the labeled points produce a vector $\sigma=(\sigma_1,...,\sigma_n)\in\{0,1\}^n$ which determine the connected component.  The number of $p_i$ with $\lambda_{p_i}=0$ is even due to integrality of the degree of the push-forward sheaf $|\theta|:=\rho_*\theta$ on the coarse curve $C$, \cite{JKVMod}.  Note that the degree of the sheaf (line bundle) $\theta$ can be half-integral on $\cc$ due to orbifold points.  On a smooth curve, the
boundary type of a spin structure is determined by an associated quadratic form applied to each of the $n$ boundary classes which vanishes since it is a homological invariant, again implying that the number of $p_i$ with $\lambda_{p_i}=0$ is even.
Each component $\overline{\modm}_{g,\sigma}^{\rm spin}$ is connected except when $|\sigma|=n$, in which case there are two connected components determined by their Arf invariant, and known as even and odd spin structures.  This follows from the case of smooth spin curves proven in \cite{NatMod}.

Restricted to $\overline{\modm}_{g,\sigma}^{\rm spin}$ the bundle $E_{g,n}$ has rank
\begin{equation}  \label{rank}
\text{rank\ }E_{g,n}=2g-2+\tfrac12(n+|\sigma|)
\end{equation}
by the following Riemann-Roch calculation.   Orbifold Riemann-Roch takes into account the representation information \begin{align*}
h^0(\theta^{\vee})-h^1(\theta^{\vee})&=1-g+\deg \theta^{\vee}-\sum_{i=1}^n\lambda_{p_i}=1-g+1-g-\tfrac12n-\tfrac12|\sigma|\\
&=2-2g-\tfrac12(n+|\sigma|).
\end{align*}
We have $h^0(\theta^{\vee})=0$ since $\deg\theta^{\vee}=1-g-\frac12n<0$, and the restriction of $\theta^{\vee}$ to any irreducible component $C'$, say of type $(g',n')$, also has negative degree, $\deg\theta^{\vee}|_{C'}=1-g'-\frac12n'<0$.  Hence $h^1(\theta^{\vee})=2g-2+\frac12(n+|\sigma|)$.  Thus $H^1(\theta^{\vee})$ gives fibres of a rank $2g-2+\frac12(n+|\sigma|)$ vector bundle.

%To determine the classes $\Theta_{g,n}$ perhaps we can use Chiodo's calculation \cite{MumTow,ChiTow} of the Chern character when $n=0$$$ch_d(E_{g})=\frac{B_{d+1}(\frac32)}{(d+1)!}\kappa_d+\frac{B_{d+1}(0)}{(d+1)!}j_{0*}(\gamma_{d-1})+\frac{B_{d+1}(\frac32)}{(d+1)!}j_{1*}(\gamma_{d-1})$$where $B_k(x)$ is the Bernoulli polynomial and $j_{(\text{irr})*}(\gamma_d)$ is defined as follows.  A family of curves $C\to S$ contains a singular set $\text{Sing}\subset C$ consisting of the nodal points in $C$.  It is naturally double-covered ... Inside the universal curve $\gamma_d=\sum_{i+j=d}(-\psi)^i\hat{\psi}^j$ for $\psi$ classes and $\hat{\psi}$ the $\psi$ classes at the two branch ... Note that $ch_0(E_{g})=\kappa_0$ as expected. The Chern classes $c_d(E_g)$, which are polynomials in the Chern characters, give relations when $d>2g-2$.  There are $\kappa$ classes present, but also boundary classes.

The forgetful map $p:\overline{\modm}_{g,n}^{\rm spin}\to\overline{\modm}_{g,n}$ and its restriction $p^{\sigma}\hspace{-1mm}:\overline{\modm}_{g,\sigma}^{\rm spin}\to\overline{\modm}_{g,n}$ forgets the spin structure and maps the twisted curve to its underlying coarse curve.

 For $\sigma\in\{0,1\}^n$, define
\[ \Omega_{g,n}^\sigma:=\epsilon^\sigma_{g,n}\cdot p^{\sigma}_*c_{\text{top}}(E_{g,n})\in H^{4g-4+n+|\sigma|}(\overline{\modm}_{g,n},\bq)\]
where $\epsilon^\sigma_{g,n}=(-1)^n2^{g-1+\frac12(n+|\sigma|)}$.  In the Neveu-Schwarz case, we have
$$\Theta_{g,n}:=\Omega_{g,n}^{(1^n)}\in H^{4g-4+2n}(\overline{\modm}_{g,n},\bq).
$$

The classes $\Omega_{g,n}^\sigma$ satisfy natural restriction properties with respect to the boundary divisors:
\[\phi_{\text{irr}}:\overline{\modm}_{g-1,n+2}\to\overline{\modm}_{g,n},\quad \phi_{h,I}:\overline{\modm}_{h,|I|+1}\times\overline{\modm}_{g-h,|J|+1}\to\overline{\modm}_{g,n}\]
for any $I\sqcup J=\{1,...,n\}$.  The restrictions, or pullbacks, insert Neveu-Schwarz behaviour at nodes, \cite{NorSup}:
\begin{equation}  \label{rest}
\phi_{\text{irr}}^*\Omega_{g,n}^\sigma=\Omega_{g-1,n+2}^{\sigma,1,1},\qquad \phi_{h,I}^*\Omega_{g,n}^\sigma=\Omega_{h,|I|+1}^{\sigma_I,1}\otimes \Omega_{g-h,|J|+1}^{\sigma_J,1}.
\end{equation}
These restriction properties are a consequence of the restriction properties of the rank two cohomological field theory obtained from the pushforward of the {\em total} Chern class $c(E_{g,n})$.  This is a special case of a more general construction in \cite{LPSZChi} of a rank $r$ CohFT obtained from the pushforwards of classes defined over moduli spaces of $r$-spin curves known as Chiodo classes.  The restrictions to boundary divisors of the rank two cohomological field theory are sums over terms with insertions of both Neveu-Schwarz and Ramond behaviour at nodes.  The Ramond insertions produce lower degree terms which are annihilated when taking the top degree term to get $\Omega_{g,n}^\sigma$.

Given any subset $I\subset\{1,...,n\}$, satisfying $|I|\geq 2$, (and $|I^c|\geq 2$ if $g=0$) define $D_I$ to be the image of $\phi_{0,I}$, i.e. the boundary divisor of nodal stable curves with a rational component containing the points $\{p_j\mid j\in I\}$.  It is represented by the stable graph:
\begin{center}

\begin{tikzpicture}[scale=0.8]
\node (1) at (0, 0) [circle,draw] {0};
\node (4) at (2, 0) [circle,draw] {$g$};
\node (2) at (-1, 1) {};
\node (3) at (-1, -1) {};
\node (5) at (3, 1) {};
\node (6) at (3, -1) {};

\node (10) at (-1, 0) {$I\ \vdots$};
\node (11) at (3, 0) {$\vdots\ I^c$};
\draw[-] (1) to (4);
\draw[-] (2) to (1);
\draw[-] (3) to (1);
\draw[-] (5) to (4);
\draw[-] (6) to (4);
\end{tikzpicture}

\end{center}
In this case \eqref{rest} becomes
\[\Omega_{g,n}^\sigma\cdot D_I=\phi_{0,I}^*\Omega_{g,n}^\sigma=\Omega_{0,|I|+1}^{\sigma_I,1}\otimes \Omega_{g,|J|+1}^{\sigma_J,1}.
\]

Define
\begin{equation}  \label{spincor}
\langle\prod_{i=1}^n\tau_{k_i}\rangle_g:=\frac{1}{m!}\int_{\overline{\modm}_{g,n+m}}\hspace{-5mm}\Omega_{g,n+m}^{(1^n,0^m)}.\prod_{i=1}^n\psi_i^{k_i},\qquad m=2-2g+2|k|
\end{equation}
where $2g-2+n+\tfrac12m+|k|=3g-3+n+m$ which gives matching of the degree of the class $\Omega_{g,n+m}^{(1^n,0^m)}\prod_{i=1}^n\psi_i^{k_i}$ with $\dim\overline{\modm}_{g,n+m}$. It uniquely determines $m$ and shows that the right hand side vanishes for $m\neq2-2g+2|k|$.   Then
\begin{align*}
Z^\Omega(s,\hbar,t_0,t_1,...)&=\exp\sum_{g,n,\vec{k}}\frac{\hbar^{g-1}}{n!}s^{2-2g+2|k|}\langle\prod_{i=1}^n\tau_{k_i}\rangle_g\prod_{i=1}^nt_{k_i}\\
&=\exp\sum_{g}\left(\frac{\hbar}{s^2}\right)^{g-1}\hspace{-2mm}\Big\langle\exp\big(\sum_{j=0}^\infty s^{2j}\tau_jt_j\big)\Big\rangle_g.
\end{align*}

It is proven in \cite{NorSup} that the spin classes satisfy the relation  
\[
\Omega_{g,n+1}^{(\sigma,1)}=\psi_{n+1}\pi^*\Omega_{g,n}^\sigma.
\]
Via the pushforward, this produces a dilaton type equation:
\begin{equation} \label{dilaton}
\Big\langle\tau_0\prod_{i=1}^n\tau_{k_i}\Big\rangle_g=(2g+n+2|k|)\Big\langle\prod_{i=1}^n\tau_{k_i}\Big\rangle_g,\quad |k|=\sum_1^nk_i
\end{equation}
or in terms of the partition function
\[\pd{}{t_0}Z^\Omega(s,\hbar,t_0,t_1,...)=\left(\sum_{i =0}^\infty (2i+1) t_i \frac{\partial}{\partial t_i}+\frac18+\frac12\hbar^{-1} s^2\right)Z^\Omega(s,\hbar,t_0,t_1,...)\]
which coincides with the first of the Virasoro constraints given in \eqref{GBGWV} satisfied by $Z^{BGW}(s,\hbar,t_0,t_1,...)$.  It is known as a dilaton equation because it implies a homogeneity condition on $Z^\Omega$, and it is sometimes also referred to as a string equation because it involves removal of a $\tau_0$ term.

%In fact, a {\em weak} version of Conjecture~\ref{KTheta}, the conjectured vanishing of all intersection numbers of tautological classes with $K_{2g-2+n}-\Theta_{g,n}$, is enough to imply the equality $Z^{\text{BGW}}=Z^\Theta$.
\subsection{Polynomials in kappa classes}
The Mumford-Morita-Miller classes \cite{MumTow} defined in the following form by Arbarello-Cornalba \cite{ACoCom}
are defined with respect to  the forgetful map $\pi:\overline{\modm}_{g,n+1}\stackrel{\pi}{\longrightarrow}\overline{\modm}_{g,n}$ by
$\kappa_m=\pi_*\psi^{m+1}_{n+1}\in H^{2m}( \overline{\cal M}_{g,n},\bq).$ 
Following \cite{KNoPol}, define degree $m$ homogeneous polynomials in $\kappa_j$, denoted $K_m\in H^{2m}(\overline{\modm}_{g,n},\bq)$, by
\[\bk(t)=K_0+K_1t+K_2t^2+...=\exp\left(\sum\sigma_i\kappa_it^i\right)\in H^*(\overline{\modm}_{g,n},\bq)[[t]],\]
where $\{\sigma_i\in\bq\mid i>0\}$ are uniquely determined by
\[\exp\left(-\sum_{i>0}\sigma_it^i\right)=\sum_{k=0}^\infty(-1)^k(2k+1)!!t^k.
\]
We also define
\[\bk:=\bk(1)=K_0+K_1+K_2+...\in H^*(\overline{\modm}_{g,n},\bq)
\]
which is a slight abuse of notation.
The first few polynomials are given by:
$$K_1=3\kappa_1,\quad K_2=\frac32(3\kappa_1^2-7\kappa_2),\quad K_3=\frac32(3\kappa_1^3-21\kappa_1\kappa_2+46\kappa_3),$$
$$K_4=\frac98(3\kappa_1^4-42\kappa_1^2\kappa_2+49\kappa_2^2+184\kappa_1\kappa_3-562\kappa_4).
$$
Define the partition function
\begin{equation}  \label{tauK}
Z^\bk(s,\hbar,t_0,t_1,\dots)=\exp\sum_{g,n}\frac{\hbar^{g-1}}{n!}s^{4g-4+2n}\hspace{-1mm}\sum_{\vec{k}\in\bn^n}\int_{\overline{\modm}_{g,n}}\hspace{-3mm}\bk(s^{-2})\prod_{i=1}^n\psi_i^{k_i}t_{k_i}.
\end{equation}
The variable $s$ in \eqref{tauK} serves as a degree operator, and hence is unnecessary and can be replaced by $s=1$.  The full function can be reconstructed from its $s=1$ specialisation via $Z^\bk(s,\hbar,t_0,t_1,\dots)=Z^\bk(s=1,\hbar s^{-2},\{t_ks^{2k}\})$.

By work of Manin and Zograf \cite{MZoInv}, $Z^\bk$ and $Z^{\text{KW}}$ are related by the following change of coordinates:
\begin{equation}  \label{KWtrans}
Z^\bk(s=1,\hbar,t_0,t_1,\dots)=
Z^{\text{KW}}(\hbar,t_0,t_1,t_2+3!!,t_3-5!!,t_4+7!!,\dots)
\end{equation}
and in particular $Z^\bk(s,\hbar,t_0,t_1,\dots)$ (and its specialisation at $s=1$) is also a tau function of the KdV hierarchy. 

The polynomials $K_m(\kappa_1,...,\kappa_m)$ have independent interest because they give the following polynomial relations among the $\kappa_j$, conjectured in \cite{KNoPol} and proven in \cite{CGGRel}.
\begin{theorem}[\cite{CGGRel}]  \label{vanK}
\[
K_m(\kappa_1,...,\kappa_m)=0\quad\text{for}\quad m>2g-2+n,\quad\text{except}\quad (m,n)=(3g-3,0).
\]
\end{theorem}

\subsection{Relation between spin and kappa classes}

Denote the forgetful map that forgets the final $m$ points  by  $\pi^{(m)}:\overline{\modm}_{g,n+m}\to\overline{\modm}_{g,n}$.   We call these $m$ points {\em Ramond} points since they  correspond to the Ramond points before forgetting the spin structure.
The following theorem gives a relation between the spin classes $\Omega_{g,n}^\sigma$ and $\bk$.
\begin{theorem}[\cite{CGGRel}]  
For $2g-2+n>0$,
\begin{equation}  \label{pushK}
\sum_{m\geq0}\frac{1}{m!}\pi^{(m)}_*\Omega_{g,n+m}^{(1^n,0^m)}=\bk-\delta_{n,0}K_{3g-3}\in H^*(\overline{\modm}_{g,n},\bq).
\end{equation}
\end{theorem}

The proof of \eqref{pushK} in \cite{CGGRel} was obtained by showing that the left hand side of \eqref{pushK} defines a semisimple rank one cohomological field theory.  Then, using the Givental-Teleman classification of semisimple cohomological field theories, the cohomological field theory is shown to be given by the right hand side of \eqref{pushK}.  The following theorem gives a generalisation of \eqref{pushK} which allows insertions of $\psi$ classes at Neveu-Schwarz points.
Define
\[ x^{(m)}:=\sum_{j=0}^m\frac{x^{m-j}}{2^jj!}.
\]
\begin{theorem}   \label{thalg}
For $2g-2+n>0$,
\begin{equation}  \label{pushf}
\sum_{m\geq0}\frac{1}{m!}\pi^{(m)}_*\Big(\Omega_{g,n+m}^{(1^n,0^m)}\prod_{j=1}^n\psi_j^{k_j}\Big)=(\bk-\delta_{n,0}K_{3g-3})\cdot\prod_{j=1}^n\psi_j^{(k_j)}\in H^*(\overline{\modm}_{g,n},\bq).
\end{equation}
\end{theorem}
\begin{proof}
We prove this by induction on $|k|=\sum k_j$.  The $|k|=0$ case of \eqref{pushf} is given by \eqref{pushK}, since $\psi_j^{(0)}=1$.  By the inductive hypothesis, assume \eqref{pushf} is true for $|k|<K\in\bn$.  Let $|k|=K$ and without loss of generality we may assume that $k_1>0$.  We need to prove
\[\pi^{(m)}_*\Big(\frac{1}{m!}\Omega_{g,n+m}^{(1^n,0^m)}\prod_{j=1}^n\psi_j^{k_j}\Big)=\Big[\bk\cdot\prod_{j=1}^n\psi_j^{(k_j)}\Big]_{2g-2+n-\frac12m+|k|}
\]
where $[\cdot]_d$ gives the degree $d$ part of a class.

Let $S$ be the set of non-empty subsets of the Ramond points $\{n+1,...,n+m\}$.  By repeated application of the pullback formula $\psi_1=\pi^*\psi_1+D_{\{1,n+1\}}$, we have:
\[\psi_1=\pi^{(m)*}\psi_1+\sum_{I\in S}D_{\{1,I\}}
\]
where the sum is over all stable graphs of the form
\begin{center}
\begin{tikzpicture}[scale=0.6]
\node (1) at (0, 0) [circle,draw] {0};
\node (4) at (2, 0) [circle,draw] {$g$};
\node (2) at (-1, 1) {};
\node (3) at (-1, -1) {};
\node (5) at (3, 1) {};
\node (6) at (3, -1) {};
\node (7) at (3.3, .4) {};

\node (10) at (-1, 0) {$\quad \vdots$};
\node  at (-1, -.4) {$I$};
\node (11) at (2.9, 0) {$\vdots$};
\node  at (2.9, -.4) {$\quad J$};
\draw[-] (1) to (4);
\draw[-] (2) to (1);
\draw[-] (3) to (1);
\draw[-] (5) to (4);
\draw[-] (6) to (4);
\draw[-] (7) to (4);

\node  at (3,.5) {$\cdot$};
\node  at (2.95,.6) {$\cdot$};

\node (12) at (-1, 1) {1};
\node (13) at (3.1, 1) {2};
\node (14) at (3.3, .4) {$n$};
\end{tikzpicture}
\end{center}
for $I\sqcup J=\{n+1,...,n+m\}$.

Set $k'_j=k_j-\delta_{1,j}$ which will facilitate removal of a factor of $\psi_1$ from $\prod_{j=1}^n\psi_j^{k_j}$.  Then
\begin{align*}
\pi^{(m)}_*\Big(\frac{1}{m!}\Omega_{g,n+m}^{(1^n,0^m)}\prod_{j=1}^n\psi_j^{k_j}\Big)&=\pi^{(m)}_*\Big(\psi_1\frac{1}{m!}\Omega_{g,n+m}^{(1^n,0^m)}\prod_{j=1}^n\psi_j^{k'_j}\Big)\\
&=\pi^{(m)}_*\Big(\big[\pi^{(m)*}\psi_1+\sum_{I\in S}D_{\{1,I\}}\big]\frac{1}{m!}\Omega_{g,n+m}^{(1^n,0^m)}\prod_{j=1}^n\psi_j^{k'_j}\Big)\\
=\psi_1&\cdot\Big[\bk\cdot\prod_{j=1}^n\psi_j^{(k'_j)}\Big]_d+\pi^{(m)}_*\sum_{I\in S}D_{\{1,I\}}\frac{1}{m!}\Omega_{g,n+m}^{(1^n,0^m)}\prod_{j=1}^n\psi_j^{k'_j}\\
\end{align*}
where $d=2g-2+n-\frac12m+|k|-1$.  The second equality uses the pullback formula for $\psi_1$ and the third equality uses the inductive assumption.  By the natural restriction properties \eqref{rest} satisfied by $\Omega_{g,n+m}^{(1^n,0^m)}$, we have
\[D_{\{1,I\}}\cdot\Omega_{g,n+m}^{(1^n,0^m)}\prod_{j=1}^n\psi_j^{k'_j}=\Omega_{0,2+|I|}^{(1^2,0^I)}\psi_1^{k_1-1}\otimes\Omega_{g,n+|J|}^{(1^n,0^J)}\prod_{j=2}^n\psi_j^{k_j}
\]
where $I\sqcup J=\{n+1,...,n+m\}$.
The pushforward of this class vanishes when $|I|\neq 2k_1$.  This is because the forgetful map forgets all but two points in the first factor, hence the pushforward acts by integration over $\overline{\modm}_{0,2+|I|}$ which vanishes unless $\deg\Big(\Omega_{0,2+|I|}^{(1^2,0^I)}\psi_1^{k_1-1}\Big)=\dim\overline{\modm}_{0,2+|I|}$.  For $|I|=2k_1$, the pushforward is given by
\[\pi^{(m)}_*\Big(\Omega_{0,2+|I|}^{(1^2,0^I)}\psi_1^{k_1-1}\cdot\Omega_{g,n+|J|}^{(1^n,0^J)}\prod_{j=2}^n\psi_j^{k_j}\Big)=(2k_1)!\langle\tau_{k_1-1}\tau_0\rangle_0\pi^{(J)}_*\Big(\Omega_{g,n+|J|}^{(1^n,0^J)}\prod_{j=2}^n\psi_j^{k_j}\Big)%,\quad |I|=2k_1
\]
where $\pi^{(J)}$ forgets the Ramond points $\{p_j\mid j\in J\}$.
In the sum over $I\in S$, each summand depends only on $|I|$.  For example, the term  $\pi^{(m)}_*\Big(\Omega_{g,n+|J|}^{(1^n,0^J)}\prod_{j=2}^n\psi_j^{k_j}\Big)$ depends only on $|J|$, by symmetry of the classes $\Omega_{g,n+|J|}^{(1^n,0^J)}$.  Hence each of the $\binom{m}{2k_1}$ summands with $|I|=2k_1$ and $|J|=m-2k_1$ 
contributes the same amount, and
\begin{align*}
\pi^{(m)}_*\sum_{I\in S}D_{\{1,I\}}\frac{1}{m!}\Omega_{g,n+m}^{(1^n,0^m)}\prod_{j=1}^n\psi_j^{k'_j}
&=\binom{m}{2k_1}\frac{(2k)!}{m!}\langle\tau_{k_1-1}\tau_0\rangle_0\pi^{(J)}_*\Big(\Omega_{g,n+|J|}^{(1^n,0^J)}\prod_{j=2}^n\psi_j^{k_j}\Big)\\
&=\langle\tau_{k_1-1}\tau_0\rangle_0\cdot\pi^{(J)}_*\Big(\frac{1}{|J|!}\Omega_{g,n+|J|}^{(1^n,0^J)}\prod_{j=2}^n\psi_j^{k_j}\Big)\\
%&=\bk\cdot\psi_1\prod_{j=1}^n\psi_j^{(k'_j)}+\frac{1}{|I|!}\langle\tau_{k_1-1}\tau_0\rangle_0\pi_*\Big(\frac{1}{|J|!}\Omega_{g,n+|J|}^{(1^n,0^J)}\prod_{j=2}^n\psi_j^{k_j}\Big)\\
&=\langle\tau_{k_1-1}\tau_0\rangle_0\cdot\Big[\bk\cdot\prod_{j=2}^n\psi_j^{(k_j)}\Big]_{2g-2+n-\frac12m+|k|}
\end{align*}
which again uses the inductive assumption, together with $|J|=m-2k_1$ which implies $2g-2+n-\frac12|J|+|k|-k_1=2g-2+n-\frac12m+|k|$.
For $|I|=2k_1$,
\[ \pi_*\Big(\tfrac{1}{|I|!}\Omega_{0,2+|I|}^{(1^2,0^{I})}\psi_1^{k_1-1}\Big)=\langle\tau_{k_1-1}\tau_0\rangle_0=\frac{(2k_1-1)!!}{(2k_1)!}=\frac{1}{2^{k_1}k_1!}.\]
Since
\[ x^{(m+1)}=x\cdot x^{(m)}+\frac{1}{2^mm!}\quad\Rightarrow\quad\prod_{j=1}^n\psi_j^{(k_j)}=\psi_1\prod_{j=1}^n\psi_j^{(k'_j)}+\frac{1}{2^{k_1}k_1!}\prod_{j=2}^n\psi_j^{(k_j)}
\]
we end up with
\begin{align*}
\pi^{(m)}_*\Big(\frac{1}{m!}\Omega_{g,n+m}^{(1^n,0^m)}\prod_{j=1}^n\psi_j^{k_j}\Big)=&\Big[\bk\cdot\psi_1\prod_{j=1}^n\psi_j^{(k'_j)}\Big]_d+\frac{1}{2^{k_1}k_1!}\Big[\bk\cdot\prod_{j=2}^n\psi_j^{(k_j)}\Big]_d\\
%\bk\cdot\Big(\psi_1\prod_{j=1}^n\psi_j^{(k'_j)}+\frac{1}{2^{k_1}k_1!}\cdot\prod_{j=2}^n\psi_j^{(k_j)}\Big)
&=\Big[\bk\cdot\prod_{j=1}^n\psi_j^{(k_j)}\Big]_{2g-2+n-\frac12m+|k|}
\end{align*}
where $d=2g-2+n-\frac12m+|k|$.  By induction we have proven this formula for all $|k|=K\in\bn$.

\end{proof}

Integrate \eqref{pushf} over $\overline{\modm}_{g,n+m}$ for $2g-2+n>0$ and $n>0$ to get,
\begin{equation}  \label{spink}
\sum_{m\geq0}\frac{1}{m!}\int_{\overline{\modm}_{g,n+m}}\hspace{-5mm}\Omega_{g,n+m}^{(1^n,0^m)}\prod_{j=1}^n\psi_j^{k_j}=\int_{\overline{\modm}_{g,n}}\bk\cdot\prod_{j=1}^n\psi_j^{(k_j)}.
\end{equation}
Define
\begin{equation}  \label{brack}
Z^{(\bk)}(\hbar,t_0,t_1,\dots):=\exp\sum_{g,n}\frac{\hbar^{g-1}}{n!}\hspace{-1mm}\sum_{\vec{k}\in\bn^n}\int_{\overline{\modm}_{g,n}}\hspace{-3mm}\bk\prod_{i=1}^n\psi_i^{(k_i)}t_{k_i}
\end{equation}
and $Z^{(\bk)}(s,\hbar,t_0,t_1,\dots):=Z^{(\bk)}(\hbar s^{-2},\{t_ks^{2k}\})$.  Then \eqref{spink} can be expressed as
\begin{equation}  \label{partspink}
Z^\Omega(s,\hbar,t_0,t_1,...)=\exp(U)\cdot Z^{(\bk)}(s,\hbar,t_0,t_1,\dots)
\end{equation}
where $U$ contains the $(g,n)=(0,1)$ and $(0,2)$ unstable terms of $\Omega$ as well as $n=0$ terms of $\bk$ to fix the normalisation at $\vec{t}=0$.  To prove Theorem~\ref{BGW=spin}, in the next section we will relate $Z^\bk$, hence also $Z^{(\bk)}$, to $Z^{BGW}$.

\section{Integrable structure}   \label{sec:integ}

\subsection{Virasoro constraints} 
The Kontsevich-Witten KdV tau function defined in \eqref{tauKW} satisfies the following Virasoro constraints \cite{WitKon}
\[
(2m+3)!!\frac{\partial}{\partial t_{m+1}}Z^{\text{KW}}(\hbar,\vec{t})=L_mZ^{\text{KW}}(\hbar,\vec{t}),\quad m=-1,0,...\]
for the Virasoro partial differential operators:
\begin{align*} 
L_m& =  \frac{1}{2}\hbar \hspace{-3mm}\mathop{\sum_{i+j=m-1}} \hspace{-3mm}(2i+1)!!(2j+1)!! \frac{\partial^2}{\partial t_i \partial t_j} +\sum_{i =0}^\infty \frac{(2i+2m+1)!!}{(2i-1)!!} t_i \frac{\partial}{\partial t_{i+m}} \\
&\qquad+ \frac{1}{8} \delta_{m,0}+\frac12\frac{t_0^2}{\hbar}\delta_{m,-1}.   \nonumber
\end{align*}
%$$L_m=\tfrac12\hbar\hspace{-2mm}\sum_{i+j=m-1}(2i+1)!!(2j+1)!!\pd{^2}{t_i\partial t_j}+
%\sum_{k-i=m}t_i\frac{(2k+1)!!}{(2i-1)!!}\pd{}{t_k}+
%\delta_{m,-1}\hbar^{-1}\frac{t_0^2}{2}+\delta_{m,0}\frac{1}{8}.$$
Using the Manin-Zograf change of coordinates relating $Z^\bk(s,\hbar,t_0,t_1,...)$, defined in \eqref{tauK}, with $Z^{KW}(\hbar,t_0,t_1,...)$ we have
$$\left(L_m+\sum_{k>m}(-1)^{m-k} s^{2(m-k)}(2k+1)!!\pd{}{t_k}\right)Z^\bk(s,\hbar,\vec{t})=0,
\quad m=-1,0,1,...$$
An immediate consequence is the following.
\begin{proposition}[\cite{KNoPol}]  \label{sectcf}
$Z^\bk(s,\hbar,t_0,t_1,\dots)$ is uniquely determined by the equations 
\begin{equation}\label{virK}
\left((2m+1)!!\pd{}{t_m}-s^2\;L_{m-1}-L_m\right)Z^\bk(s,\hbar,\vec{t})=0,\quad m=0,1,2,...
\end{equation}
and the value of $Z^\bk(s,\hbar,\vec{t}=0)$.  Furthermore, $Z^\bk(s,\hbar,\vec{t})/Z^\bk(s,\hbar,\vec{t}=0)$ is regular in $s$.
\end{proposition}

The value $Z^\bk(s,\hbar,\vec{t}=0)$ is described by \cite[Corollary C]{CGGRel} conjectured earlier in \cite{KNoPol}. Namely, for $g\geq 2$ we have
\[\int_{\overline{\modm}_{g}}\bk=-|\chi(\modm_g)|= (-1)^g \frac{B_{2g}}{2g(2g-2)},\]
%\[\int_{\overline{\modm}_{g}}\hspace{-3mm}\bk(s^{-2})=-|\chi(\modm_g)|s^{6-6g}= (-1)^g \frac{B_{2g}}{2g(2g-2)}s^{6-6g},\]
%\[\int_{\overline{\modm}_{g}}\hspace{-3mm}\bk(s^{-2})=-|\chi(\modm_g)|s^{6-6g}= (-1)^g \frac{B_{2g}}{2g(2g-2)}s^{6-6g},\]
where $\chi(\modm_g)$ is the orbifold Euler characteristic of the moduli space of genus $g$ curves and $B_{2g}$ is a Bernoulli number.  Denote
\[\cn(\hbar)=\sum_{g\geq2}|\chi(\modm_g)|\hbar^{g-1}
\]  
so that
\[
Z^\bk(s,\hbar,\vec{t}=0)=e^{-\cn(\hbar s^{-2})}.
\]

The BGW tau function is uniquely determined, up to a constant, by the following Virasoro constraints which are shifted relative to the Kontsevich-Witten KdV tau function constraints \cite{GNeUni}:
\[\left((2m+1)!!\pd{}{t_m}-L_m\right)Z^{BGW}(s=0,\hbar,t_0,t_1,\dots)=0,\quad m=0,1,2,\dots.\]
These Virasoro constraints coincide with \eqref{virK} specialised at $s=0$.  Hence:
\begin{proposition}[\cite{KNoPol}]
$\left.e^{\cn(\hbar s^{-2})}Z^\bk(s,\hbar,\vec{t})\right|_{s=0}=Z^{BGW}(s=0,\hbar, \vec{t})$.
\end{proposition}
 
%The $s$-rescaled potential for intersection numbers of the class $\bk$ with $\psi$ monomials, with the constant term removed, admits a limit for $s\to 0$. The limit is a similar potential enumerating intersection numbers of $\psi$ monomials with the component $K_{2g-2+n}$ of $\bk$. Denote this limit by $F^\bk$, $$F^\bk(\hbar,t_0,t_1,...)=\sum_{g,n,\vec{k}}\frac{\hbar^{g}}{n!}\int_{\overline{\modm}_{g,n}}K_{2g-2+n}\cdot\prod_{j=1}^n\psi_j^{k_j}t_{k_j}.$$

\subsection{Conjugation of the string equation}

The tau function $Z^{BGW}(s,\hbar,\vec{t})$ of the generalized BGW model satisfies the Virasoro constraints
\begin{equation}\label{GBGWV}
\left((2m+1)!!\pd{}{t_m}-L_m-\frac12\hbar^{-1} s^2\delta_{m,0}\right)Z^{BGW}(s,\hbar,t_0,t_1,\dots)=0,\quad m\geq0
\end{equation}
which differ from those of its $s=0$ specialisation only when $m=0$.

Here we find a relation between the string equation ($m=0$) of these constraints and the $m=0$ case of \eqref{virK}. From the commutation relations between the Virasoro operators $[L_{-1},L_0]=-2L_{-1}$ we have
\begin{equation}  \label{conj1}
e^{\frac{s^2}{2} L_{-1}} (L_0 + s^2 L_{-1}) e^{-\frac{s^2}{2} L_{-1}}=L_0.
\end{equation}
\begin{lemma}
The following conjugation holds:
\begin{equation}   \label{conj2}
e^{\frac{s^2}{2} L_{-1}} \pd{}{t_0} e^{-\frac{s^2}{2} L_{-1}}=\pd{}{t_0}-\hbar^{-1}\sum_{m=0}^\infty \alpha_m t_m,\qquad \alpha_m:=\frac{(s^2/2)^{m+1}}{(m+1)!}.
\end{equation}
\end{lemma}
\begin{proof}
The proof is by induction.
The coefficient of $(s^2/2)^k/k!$ in
\[e^{\frac{s^2}{2} L_{-1}} \pd{}{t_0}= \pd{}{t_0}e^{\frac{s^2}{2} L_{-1}}-\hbar^{-1}\sum_{m=0}^\infty\alpha_m t_me^{\frac{s^2}{2} L_{-1}}
\]
with the proposed values of $\alpha_m$ is
\[ L_{-1}^k\pd{}{t_0}=\pd{}{t_0}L_{-1}^k-\hbar^{-1}\sum_{m=0}^{k-1}\binom{k}{m+1}t_m(L_{-1})^{k-1-m}
\]
which holds for $k=1$ since $L_{-1}\pd{}{t_0}=\pd{}{t_0}L_{-1}-\hbar^{-1}t_0$.  Apply $L_{-1}$ to both sides to get:
\begin{align*} L_{-1}^{k+1}\pd{}{t_0}&=L_{-1}\pd{}{t_0}L_{-1}^{k}-\hbar^{-1}L_{-1}\sum_{m=0}^{k-1}\binom{k}{m+1}t_mL_{-1}^{k-1-m}\\
&=\pd{}{t_0}L_{-1}^{k+1}-\hbar^{-1}t_0L_{-1}^k-\hbar^{-1}\sum_{m=0}^{k-1}\binom{k}{m+1}(t_mL_{-1}+t_{m+1})L_{-1}^{k-1-m}\\
&=\pd{}{t_0}L_{-1}^{k+1}-\hbar^{-1}\sum_{m=0}^{k}\binom{k+1}{m+1}t_mL_{-1}^{k-m}
\end{align*}
which uses $[L_{-1},t_m]=t_{m+1}$ and $\binom{k}{m}+\binom{k}{m+1}=\binom{k+1}{m+1}$. 
\end{proof}

From \eqref{conj1} and \eqref{conj2}
\begin{equation} \label{conj3}
e^{\frac{s^2}{2} L_{-1}}   \left(\pd{}{t_0} -s^2 L_{-1} -L_0 \right) e^{-\frac{s^2}{2} L_{-1}}=\pd{}{t_0}-\hbar^{-1}\sum_{m=0}^\infty \alpha_m t_m -L_0 
\end{equation}
as required.

For a general sequence $\{a_m\}$, define $\displaystyle S_a(t_0,t_1,...)=\sum_{m=0}^\infty\frac{a_mt_m}{2m+1}$.  Then the relation
\begin{equation} \label{conj4}
e^{\hbar^{-1}S_a} \left(\pd{}{t_0}-\hbar^{-1}\sum_{m=0}^\infty a_m t_m -L_0 \right) e^{-\hbar^{-1}S_a}=\pd{}{t_0}-L_0 -\hbar^{-1}a_0
\end{equation}
is immediate using logarithmic derivatives.
When $\{a_m\}=\{\alpha_m\}$, the right hand side of \eqref{conj4} becomes $\pd{}{t_0}-L_0 -\frac12\hbar^{-1}s^2$. Hence \eqref{conj3} and \eqref{conj4} show that the string equation of the generalized BGW tau function and that of $Z^\bk$ are related by conjugation by the operator $e^{\hbar^{-1}S_\alpha}e^{\frac12s^2L_{-1}}$  hence also by conjugation by the operator 
\begin{equation}   \label{opD}
\cd=e^{\cn(\hbar s^{-2})}e^{\hbar^{-1}S_\alpha}e^{\frac12s^2L_{-1}}.
\end{equation}
So we have:
\[\cd\left(\pd{}{t_0}-\hbar^{-1}\sum_{m=0}^\infty \alpha_m t_m -L_0 \right) \cd^{-1}
=\pd{}{t_0}-L_0 -\frac12\hbar^{-1}s^2.\]
Thus the operator $\cd$ gives the following  relation between $Z^{BGW}$ and $Z^\bk$. 
\begin{proposition}   \label{BGWK}
\begin{equation}   \label{eq:BGWK}
Z^{BGW}(s,\hbar,t_0,t_1,\dots)=\cd\cdot Z^\bk(s,\hbar,t_0,t_1,\dots)
\end{equation}
\end{proposition}
\begin{proof}
The operator $\cd$ is a reduction of an operator from the Heisenberg-Virasoro subgroup of the $GL(\infty)$ group of symmetries of the KP hierarchy to the space of functions of odd (KdV) variables. Therefore, the
operator is a symmetry operator of the KdV hierarchy, and always transforms a KdV tau function to a KdV hierarchy. 

Therefore, the function $\cd \cdot Z^\bk(s,\hbar,\vec{t})$ is a tau function of the KdV hierarchy, which satisfies the string equation \eqref{GBGWV} together with the initial value $Z(s,\hbar,\vec{t})|_{\vec{t}=0}=e^{\cn(\hbar s^{-2})}$.   Such a function is unique. Therefore, we have proven Proposition~\ref{BGWK}.
\end{proof}

The relation between the generalised Br\'ezin-Gross-Witten tau function for $s\neq 0$ and the Witten-Kontsevich tau function with shifted arguments was derived, up to a normalization term $e^{\cn(\hbar s^{-2})}$, by Ambj{\o}rn and Chekhov \cite{AC} by a similar argument based on the conjugation of the Virasoro constraints. 
The full identification including the normalization terms was obtained recently by Xu and Yang in \cite[Theorem 3]{XYaGal} using a different approach.

The group element can be further factorized. %by the Baker-Campbell-Hausdorff formula.  
\begin{proposition}
\begin{equation}\label{Virgroup} 
 e^{\frac{s^2}{2} L_{-1}} = e^{\frac{\hbar^{-1}}{2}F_{0,2}}e^{\frac{s^2}{2} \sum_{k=0}^\infty t_{k+1}\frac{\partial}{\partial t_k}}
 \end{equation}
 for
 \[ F_{0,2}=\sum_{m_1,m_2}\frac{1}{m_1+m_2+1}\frac{t_{m_1}t_{m_2}}{m_1!m_2!}\left(\frac{s^2}{2}\right)^{m_1+m_2+1}.
\]
\end{proposition}
\begin{proof}
The factorization rule of the Virasoro group is well-known, but
let us give here a proof for completeness. The structure of the right hand side of \eqref{Virgroup} follows from the expression of the Virasoro operator, it remains to prove the expression for $F_{0,2}$. Let us act by both sides of $\eqref{Virgroup}$ on $1$: 
\begin{equation}\label{F02d}
 e^{\frac{s^2}{2} L_{-1}} \cdot 1 = e^{\frac{\hbar^{-1}}{2}F_{0,2}}.
 \end{equation}
Differentiating both sides of this identity by $s^2/2$ we have
 \[
L_{-1} e^{\frac{s^2}{2} L_{-1}} \cdot 1 = e^{\frac{\hbar^{-1}}{2}F_{0,2}} \frac{\hbar^{-1}}{2} \frac{1}{s}\frac{\partial}{\partial s}F_{0,2},
\]
and, combining it with \eqref{F02d} we arrive an the equation for $F_{0,2}$,
 \[
L_{-1} e^{\frac{\hbar^{-1}}{2}F_{0,2}} = e^{\frac{\hbar^{-1}}{2}F_{0,2}} \frac{\hbar^{-1}}{2} \frac{1}{s}\frac{\partial}{\partial s}F_{0,2}.
\]
This equation has a unique solution such that $\left.F_{0,2}\right|_{s=0}=0$, and this solution is given by the statement of the proposition. This completes the proof. 
\end{proof}

\subsection{Proof of Theorem~\ref{BGW=spin}} 
We now put all of the results of this section together with Theorem~\ref{thalg} to prove Theorem~\ref{BGW=spin}.

Via the substitution
\[ \hbar\mapsto\hbar s^{-2},\quad t_k\mapsto t_ks^{2k}
\] 
one can reconstruct the functions and operators in this section from their $s=1$ specialisations, as already mentioned for $Z^\bk$.  Each of the expressions \eqref{conj1}, \eqref{conj2}, \eqref{conj3}, \eqref{conj4}, \eqref{opD}, \eqref{eq:BGWK}, \eqref{Virgroup} is equivalent to its $s=1$ specialisation.  In the following, we will set $s=1$ to simplify formulae.

The operator $R=e^{\frac{1}{2} \sum_{k=0}^\infty t_{k+1}\frac{\partial}{\partial t_k}}$ %$\displaystyle R=\exp\left(\frac12 \sum_{k=0}^\infty t_{k+1}\frac{\partial}{\partial t_k}\right)$ 
that appears in the factorisation \eqref{Virgroup}, acts by a linear change of variables 
\begin{equation} \label{shact}
R[f(\{t_k\})]=f(\{R(t_k)\})=f(\{t_k+\tfrac12t_{k+1}+...+\tfrac{1}{m! 2^m}t_{k+m}+...\}).
\end{equation}
This is because it acts as a ring homomorphism, i.e. $R(fg)=R(f)R(g)$ on functions $f=f(t_0,t_1,...)$ and $g=g(t_0,t_1,...)$, which is a special case of the following more general fact.  Given a first order differential operator $D=\sum_{k\geq 0} p_k(t)\frac{\partial}{\partial t_k}$ for any sequence $p_k(t)=p_k(t_0,t_1,...)$ the operator $\exp D$ acts as a ring homomorphism on formal power series in $\{t_0,t_1,...\}$ due to the following:
\begin{align*}
\exp(D)(fg)&=fg+D(fg)+\frac12D^2(fg)+...\\
&=fg+Df\cdot g+f\cdot Dg+\frac12(D^2f\cdot g+2Df\cdot Dg+f\cdot D^2g)+...\\
&=(f+Df+\frac12D^2f+...)(g+Dg+\frac12D^2g+...)
\end{align*}
where the final equality uses %$D^{m+n}(fg)=\binom{m+n}{m}D^mf\cdot D^n g+...$  
$\displaystyle\tfrac{1}{M!}D^M(fg)=\hspace{-3mm}\sum_{m+n=M}\hspace{-2mm}\tfrac{1}{m!n!}D^mf\cdot D^n g$.

As we specialise to $s=1$, we write $Z^\bk(\hbar,t_0,t_1,\dots):=Z^\bk(1,\hbar,t_0,t_1,\dots)$.  Using the action \eqref{shact}, we have
\begin{equation}  \label{KtoKbrack}
e^{\frac12L_{-1}} \cdot Z^\bk(\hbar,\{t_k\})=Z^\bk(\hbar,\{t_k+\tfrac12t_{k+1}+\tfrac18t_{k+2}+...\})=Z^{(\bk)}(\hbar,\{t_k\})
\end{equation}
for $Z^{(\bk)}$ defined in \eqref{brack}.

\begin{proof}[Proof of Theorem~\ref{BGW=spin}]
Using the factorisations of $\cd$ given in \eqref{opD} and \eqref{Virgroup}, specialised to $s=1$ so that
$S_\alpha=\sum\frac{2^{-(m+1)}t_m}{(2m+1)(m+1)!}$ and $F_{0,2}=\sum\frac{2^{-(m_1+m_2+1)}}{m_1+m_2+1}\frac{t_{m_1}t_{m_2}}{m_1!m_2!}$, we have:
\begin{align*}  
\cd\cdot Z^\bk(\hbar,t_0,t_1,\dots)&=e^{\cn(\hbar )}e^{\hbar^{-1}S_\alpha}e^{\frac12L_{-1}} \cdot Z^\bk(\hbar,t_0,t_1,\dots)\\
&=e^{\cn(\hbar)}e^{\hbar^{-1}S_\alpha}e^{\frac{\hbar^{-1}}{2}F_{0,2}}e^{\frac{1}{2} \sum t_{k+1}\frac{\partial}{\partial t_k}} \cdot Z^\bk(\hbar,t_0,t_1,\dots)\\
&=e^{\cn(\hbar)}e^{\hbar^{-1}S_\alpha}e^{\frac{\hbar^{-1}}{2}F_{0,2}}\cdot Z^{(\bk)}(\hbar,t_0,t_1,\dots).
\end{align*}  
The first and second equalities use \eqref{opD} and \eqref{Virgroup} and the third equality uses \eqref{KtoKbrack}.  

By Theorem~\ref{thalg}, and in particular its consequence \eqref{partspink} on tau functions,
the final expression above coincides with $Z^\Omega(1,\hbar,t_0,t_1,\dots)$.  This is because the stable terms of $Z^\Omega(1,\hbar,t_0,t_1,\dots)$ are given by $Z^{(\bk)}(\hbar,t_0,t_1,\dots)$ due to Theorem~\ref{thalg}.  Then we use the spin class intersection calculations, \eqref{1point} and \eqref{2point} from Section~\ref{calc}, of the $(0,1)$ terms given by $S_\alpha$ and the $(0,2)$ terms given by $F_{0,2}$.  The term $e^{\cn(\hbar)}$ arranges that evaluation at $\vec{t}=0$ is 1.  Thus
\[\cd\cdot Z^\bk(\hbar,t_0,t_1,\dots)=Z^\Omega(1,\hbar,t_0,t_1,\dots)\]
and using the substitution $\hbar\mapsto\hbar s^{-2}$ and $t_k\mapsto t_ks^{2k}$ we get
\[\label{KeqOm}
\cd\cdot Z^\bk(s,\hbar,t_0,t_1,\dots)=Z^\Omega(s,\hbar,t_0,t_1,\dots).
\]
Finally, Proposition~\ref{BGWK} which gives
$Z^{BGW}(s,\hbar,t_0,t_1,\dots)=\cd\cdot Z^\bk(s,\hbar,t_0,t_1,\dots)$
produces
\[
Z^{BGW}(s,\hbar,t_0,t_1,\dots)=Z^\Omega(s,\hbar,t_0,t_1,\dots)
\]
as required.
\end{proof}

One can prove the genus 0 part of Theorem~\ref{BGW=spin} in a different way using genus 0 topological recursion relations, which we denote TRR so as not to confuse this with spectral curve topological recursion discussed in Section~\ref{TR}.  The TRR gives an effective way to calculate the genus 0 correlators.
Define
\[ \langle \prod_{j=1}^n\tau_{(k_j)}\rangle_g^\bk:=\int_{\overline{\modm}_{g,n}}\hspace{-2mm}\bk\cdot\prod_{j=1}^n\psi_j^{(k_j)}
\]
and define a genus 0, 2-point correlator
\[ \langle \tau_{(k)}\tau_0\rangle_0^\bk:=\frac{1}{2^{k+1}(k+1)!}.
\]
Note that $\tau_{(0)}=\tau_0$ since $\psi^{(0)}=1$.
\begin{proposition}   \label{KTRR}
The genus 0 correlators $\langle \prod_{j=1}^n\tau_{(k_j)}\rangle_0^\bk$ satisfy the following TRR.  For $k_1>0$,
\begin{equation}   \label{TRRkap}
\langle \prod_{j=1}^n\tau_{(k_j)}\rangle_0^\bk=%\frac{\langle\tau_{0} \prod_{j=2}^n\tau_{(k_j)}\rangle_0^\bk}{2^{k_1}k_1!}+
\hspace{-4mm}\sum_{ I\subset\{4,...,n\}}\hspace{-3mm}\langle \tau_{(k_1-1)}\tau_0\prod_{j\in I}\tau_{(k_j)}\rangle_0^\bk
\langle\tau_{(k_2)} \tau_{(k_3)}\tau_{0} \prod_{j\in I^c}\tau_{(k_j)}\rangle_0^\bk.
\end{equation} 
\end{proposition}
\begin{proof}
In genus 0 we have, for $n\geq4$
\begin{center}
\begin{tikzpicture}[scale=0.6]

\node  at (-5.5, 0) {$\displaystyle\psi_1=\hspace{-3mm}\sum_{\emptyset\neq I\subset\{4,...,n\}}\hspace{-3mm}D_{\{1,I\}}=\hspace{-3mm}\sum_{\emptyset\neq I\subset\{4,...,n\}}$};

\node (1) at (0, 0) [circle,draw] {0};
\node (4) at (2, 0) [circle,draw] {$0$};
\node (2) at (-1, 1) {};
\node (3) at (-1, -1) {};
\node (5) at (3, 1) {};
\node (6) at (3, -1) {};
\node (7) at (2.3, 1.3) {};
\node (8) at (3.3, .4) {};

\node (10) at (-1, 0) {$\quad \vdots$};
\node  at (-1.1, -.25) {$I$};
\node (11) at (2.9, 0) {$\vdots$};
\node  at (3, -.25) {$\ \quad I^c$};
\draw[-] (1) to (4);
\draw[-] (2) to (1);
\draw[-] (3) to (1);
\draw[-] (5) to (4);
\draw[-] (6) to (4);
\draw[-] (7) to (4);
\draw[-] (8) to (4);
%\node  at (3,.5) {$\cdot$};
%\node  at (2.95,.6) {$\cdot$};

\node (12) at (-1, 1) {1};
\node (13) at (2.2, 1.4) {2};
\node (13) at (3.1, 1) {3};
\node (14) at (3.3, .4) {};
\node (15) at (-1.2, .4) {};
\draw[-] (15) to (1);
\end{tikzpicture}
\end{center}
which is a special case of the pullback formula used in the proof of Theorem~\ref{thalg}. % It is expressed here as a sum over all stable curves represented by stables graphs. 
Notice that the above sum requires non-empty $I$, whereas the statement of the theorem allows empty $I$.

The class $\bk$ restricts naturally to boundary divisors.  In particular:
\begin{center}

\begin{tikzpicture}[scale=0.8]
\node (1) at (0, 0) [circle,draw] {$g_1$};
\node (4) at (2, 0) [circle,draw] {$g_2$};
\node (2) at (-1, 1) {};
\node (3) at (-1, -1) {};
\node (5) at (3, 1) {};
\node (6) at (3, -1) {};

\filldraw  (-1.5,0) circle (1pt);
%\node  at (-1.5, 0) [circle,draw,fill,radius=1mm] {};
\node at (-2, 0) {$\bk$};
\node (10) at (-1, 0) {$I\ \vdots$};
\node (11) at (3, 0) {$\vdots\ I^c$};
\draw[-] (1) to (4);
\draw[-] (2) to (1);
\draw[-] (3) to (1);
\draw[-] (5) to (4);
\draw[-] (6) to (4);
\end{tikzpicture}
\begin{tikzpicture}[scale=0.8]
\node (1) at (0, 0) [circle,draw] {$g_1$};
\node (4) at (2, 0) [circle,draw] {$g_2$};
\node (2) at (-1, 1) {};
\node (3) at (-1, -1) {};
\node (5) at (3, 1) {};
\node (6) at (3, -1) {};

\node at (0, .8) {$\bk$};
\node at (2, .8) {$\bk$};
%\filldraw  (-1.5,0) circle (1pt);
%\node  at (-1.5, 0) [circle,draw,fill,radius=1mm] {};
\node at (-2, 0) {$=$};
\node (10) at (-1, 0) {$I\ \vdots$};
\node (11) at (3, 0) {$\vdots\ I^c$};
\draw[-] (1) to (4);
\draw[-] (2) to (1);
\draw[-] (3) to (1);
\draw[-] (5) to (4);
\draw[-] (6) to (4);
\end{tikzpicture}
\end{center}
where nodes are decorated by polynomials in kappa classes.  In the diagram above, on the left hand side $\bk\in H^*(\overline{\modm}_{g,n},\bq)$, and on the right hand side the left decorated node uses the class $\bk\in H^*(\overline{\modm}_{g_1,|I|+1},\bq)$.  The diagram represents a pushforward class into  $H^*(\overline{\modm}_{g,n},\bq)$, however we are concerned here only with top intersection classes so we consider intersections in $H^*(\overline{\modm}_{g_1,|I|+1},\bq)$.

Thus, for $k'_j=k_j-\delta_{1,j}$,
\begin{align*}
\int_{\overline{\modm}_{0,n}}\bk\cdot\prod_{j=1}^n\psi_j^{k_j}&=\int_{\overline{\modm}_{0,n}}\bk\cdot\psi_1\prod_{j=1}^n\psi_j^{k'_j}\\
&=\int_{\overline{\modm}_{0,n}}\bk\cdot\hspace{-3mm}\sum_{\emptyset\neq I\subset\{4,...,n\}}\hspace{-3mm}D_{\{1,I\}}
\prod_{j=1}^n\psi_j^{k'_j}\\
&=\hspace{-4mm}\sum_{\emptyset\neq I\subset\{4,...,n\}}\langle \tau_{k_1-1}\tau_{0}\prod_{j\in I}\tau_{k_j}\rangle_0^\bk
\langle \tau_{k_2} \tau_{k_3}\tau_{0}\prod_{j\in I^c}\tau_{k_j}\rangle_0^\bk
\end{align*}
where we have written
$\displaystyle \langle \prod_{j=1}^n\tau_{k_j}\rangle_0^\bk:=\int_{\overline{\modm}_{0,n}}\hspace{-2mm}\bk\cdot\prod_{j=1}^n\psi_j^{k_j}$.  Linearity allows us to replace each factor $\psi_j^{k_j}$ with $\psi_j^{(k_j)}=\psi_j^{k_j}+...$, to get:
\[
\int_{\overline{\modm}_{0,n}}\bk\cdot\psi_1\prod_{j=1}^n\psi_j^{(k'_j)}\\
=\hspace{-4mm}\sum_{\emptyset\neq I\subset\{4,...,n\}}\langle \tau_{(k_1-1)}\tau_{0}\prod_{j\in I}\tau_{(k_j)}\rangle_0
\langle \tau_{(k_2)} \tau_{(k_3)}\tau_{0}\prod_{j\in I^c}\tau_{(k_j)}\rangle_0.
\]
Since $\psi_1^{(k_1)}=\psi_1\cdot \psi_1^{(k_1-1)}+\frac{1}{2^{k_1}k_1!}$ we see that the extra term of $\frac{1}{2^{k_1}k_1!}$ contributes
\[\langle \tau_{(k)}\tau_0\rangle_0^\bk\langle\tau_{0} \prod_{j=2}^n\tau_{(k_j)}\rangle_0^\bk=\frac{1}{2^{k_1}k_1!}\langle\tau_{0} \prod_{j=2}^n\tau_{(k_j)}\rangle_0^\bk\]
to the sum, which corresponds to $I=\emptyset$.
\end{proof}

Recall the genus 0 spin class correlators, defined in \eqref{spincor} by:
\[\langle\prod_{i=1}^n\tau_{k_i}\rangle_0:=\frac{1}{m!}\int_{\overline{\modm}_{g,n+m}}\hspace{-5mm}\Omega_{g,n+m}^{(1^n,0^m)}.\prod_{i=1}^n\psi_i^{k_i},\quad m=2+2|k|\]
\begin{proposition}
The genus 0 spin class correlators satisfy the following genus 0 TRR.  For $k_1>0$,
\begin{equation}   \label{TRRspin}
\langle \prod_{j=1}^n\tau_{k_j}\rangle_0=\hspace{-4mm}\sum_{I\subset\{4,...,n\}}\hspace{-4mm}\langle \tau_{0}\tau_{k_1-1}\prod_{j\in I}\tau_{k_j}\rangle_0
\langle\tau_{0} \tau_{k_2} \tau_{k_3}\prod_{j\in I^c}\tau_{k_j}\rangle_0.
\end{equation} 
\end{proposition}
\begin{proof}
This proof is similar to the proof of Proposition~\ref{KTRR}, although there are now an extra $m$ marked points, and the spin class $\Omega_{g,n+m}^{(1^n,0^m)}$ is homogeneous whereas $\bk$ is not.  We have:
\begin{center}
\begin{tikzpicture}[scale=0.6]

\node  at (-6.5, 0) {$\displaystyle\psi_1=\hspace{-3mm}\sum_{\emptyset\neq I\subset\{4,...,n+m\}}\hspace{-3mm}D_{\{1,I\}}=\hspace{-3mm}\sum_{\emptyset\neq I\subset\{4,...,n+m\}}$};

\node (1) at (0, 0) [circle,draw] {0};
\node (4) at (2, 0) [circle,draw] {$0$};
\node (2) at (-1, 1) {};
\node (3) at (-1, -1) {};
\node (5) at (3, 1) {};
\node (6) at (3, -1) {};
\node (7) at (2.3, 1.3) {};
\node (8) at (3.3, .4) {};

\node (10) at (-1, 0) {$\quad \vdots$};
\node  at (-1.1, -.25) {$I$};
\node (11) at (2.9, 0) {$\vdots$};
\node  at (3, -.25) {$\ \quad I^c$};
\draw[-] (1) to (4);
\draw[-] (2) to (1);
\draw[-] (3) to (1);
\draw[-] (5) to (4);
\draw[-] (6) to (4);
\draw[-] (7) to (4);
\draw[-] (8) to (4);
%\node  at (3,.5) {$\cdot$};
%\node  at (2.95,.6) {$\cdot$};

\node (12) at (-1, 1) {1};
\node (13) at (2.2, 1.4) {2};
\node (13) at (3.1, 1) {3};
\node (14) at (3.3, .4) {};
\node (15) at (-1.2, .4) {};
\draw[-] (15) to (1);
\end{tikzpicture}
\end{center}
\iffalse
\begin{center}

\begin{tikzpicture}[scale=0.8]
\node (1) at (0, 0) [circle,draw] {$0$};
\node (4) at (2, 0) [circle,draw] {$0$};
\node (2) at (-1, 1) {};
\node (3) at (-1, -1) {};
\node (5) at (3, 1) {};
\node (6) at (3, -1) {};

\filldraw  (-1.5,0) circle (1pt);
%\node  at (-1.5, 0) [circle,draw,fill,radius=1mm] {};
\node at (-2.4, 0) {$\Omega_{0,n+m}^{(1^n,0^m)}$};
\node (10) at (-1, 0) {$I\ \vdots$};
\node (11) at (3, 0) {$\vdots\ I^c$};
\draw[-] (1) to (4);
\draw[-] (2) to (1);
\draw[-] (3) to (1);
\draw[-] (5) to (4);
\draw[-] (6) to (4);
\end{tikzpicture}
\begin{tikzpicture}[scale=0.8]
\node (1) at (0, 0) [circle,draw] {$0$};
\node (4) at (2, 0) [circle,draw] {$0$};
\node (2) at (-1, 1) {};
\node (3) at (-1, -1) {};
\node (5) at (3, 1) {};
\node (6) at (3, -1) {};

\node at (0.5, .8) {$\Omega_{0,|I|+1+m_1}^{(1^{|I|+1},0^{m_1})}$};
\node at (1.4, -.8) {$\Omega_{0,|I^c|+1+m_2}^{(1^{|I^c|+1},0^{m_2})}$};
%\filldraw  (-1.5,0) circle (1pt);
%\node  at (-1.5, 0) [circle,draw,fill,radius=1mm] {};
\node at (-2, 0) {$=$};
\node (10) at (-1, 0) {$I\ \vdots$};
\node (11) at (3, 0) {$\vdots\ I^c$};
\draw[-] (1) to (4);
\draw[-] (2) to (1);
\draw[-] (3) to (1);
\draw[-] (5) to (4);
\draw[-] (6) to (4);
\end{tikzpicture}
\end{center}
\fi
Hence
\begin{align*}
\frac{1}{m!}\int_{\overline{\modm}_{0,n+m}}\hspace{-5mm}\Omega_{0,n+m}^{(1^n,0^m)}\cdot\prod_{j=1}^n\psi_j^{k_j}&=\frac{1}{m!}\int_{\overline{\modm}_{0,n+m}}\hspace{-5mm}\Omega_{0,n+m}^{(1^n,0^m)}\cdot\hspace{-5mm}\sum_{\emptyset\neq I\subset\{4,...,n+m\}}\hspace{-5mm}D_{\{1,I\}}
\prod_{j=1}^n\psi_j^{k'_j}\\
&=\hspace{-2mm}\sum_{I\subset\{4,...,n\}}\hspace{-2mm}\langle \tau_{k_1-1}\tau_{0}\prod_{j\in I}\tau_{k_j}\rangle_0
\langle \tau_{k_2} \tau_{k_3}\tau_{0}\prod_{j\in I^c}\tau_{k_j}\rangle_0.
\end{align*}
To explain the last equality, note that the homogeneity of $\Omega_{0,n+m}^{(1^n,0^m)}$ and its restrictions to boundary divisors uniquely determines the number of Ramond points.  So
 $\langle \tau_{k_1-1}\tau_{0}\prod_{j\in I}\tau_{k_j}\rangle_0$ and $\langle \tau_{k_2} \tau_{k_3}\tau_{0}\prod_{j\in I^c}\tau_{k_j}\rangle_0$ use $m_1=2+2(k_1-1+|k_I|)$, respectively $m_2=2(k_2+k_3+|k_{I^c}|)$, Ramond points and $m=m_1+m_2$.  There are $\binom{m}{m_1}$ summands on the right hand side corresponding to choosing $m_1$ Ramond points from the $m$ Ramond points, hence the last equality uses the factor $\frac{1}{m!}\binom{m}{m_1}=\frac{1}{m_1!m_2!}$ to produce the correlators on the right hand side.  Note that since $m_1>0$, so that Ramond points will always be present in the first factor, it is possible that $I=\emptyset$.  Thus we have proven \eqref{TRRspin} as required.
\end{proof}

The TRR uniquely determines $\langle \prod_{j=1}^n\tau_{(k_j)}\rangle_0^\bk$ from $\langle\tau_0\tau_0\tau_0\rangle^\bk_0=1$ together with the two-point correlators $\langle \tau_{(k)}\tau_0\rangle_0^\bk=\frac{1}{2^{k+1}(k+1)!}$.  The same TRR uniquely determines $\langle \prod_{j=1}^n\tau_{k_j}\rangle_0$ from $\langle\tau_0\tau_0\tau_0\rangle_0=1$ together with the two-point correlators $\langle \tau_{k}\tau_0\rangle_0=\frac{1}{2^{k+1}(k+1)!}$.  Hence for $n\geq 3$,
\[\langle \prod_{j=1}^n\tau_{k_j}\rangle_0=\langle \prod_{j=1}^n\tau_{(k_j)}\rangle_0^\bk
\]
which gives a different proof of the stable genus 0 part of Theorem~\ref{BGW=spin}.  The unstable part is dealt with in the same way as the proof of Theorem~\ref{BGW=spin}.

\subsection{Super volumes}
The function 
\[\widehat{V}_{g,n}\WP(s,L_1,...,L_n)\in\br[L_1,...,L_n][[s^2]]\] 
is defined by
\[\widehat{V}_{g,n}\WP(s,L_1,...,L_n):=\sum_{m=0}^\infty\frac{s^m}{m!}\int_{\overline{\modm}_{g,n+m}}\hspace{-4mm}\Omega_{g,n+m}^{(1^n,0^m)}\exp(2\pi^2\kappa_1+\frac12\sum_{i=1}^nL_i^2\psi_i).\]
As described in the introduction, it represents super Weil-Petersson volumes of the moduli spaces of super Riemann surfaces.  In this paper we use only intersection theory on $\overline{\modm}_{g,n+m}$ without the need for supergeometry.

\begin{proposition}[\cite{NorEnu}]   \label{equivconst}
The recursion \eqref{reconj} for $\widehat{V}_{g,n}\WP(s,L_1,...,L_n)$ is equivalent to the Virasoro constraints
\begin{equation}   \label{virspin}
\left((2m+1)!!\pd{}{t_m}-L_m-\frac12\hbar^{-1} s^2\delta_{m,0}\right)Z^\Omega(s,\hbar,t_0,t_1,\dots)=0,\quad m\geq0.
\end{equation}
\end{proposition}
\begin{proof}
The proof in \cite{NorEnu} considers only the $s=0$ case, and generalises immediately for general $s$.  We describe the constructions here, emphasising those places where the $s=0$ proof needs to be adjusted.  For $b=(b_1,b_2,....)\in\br^\bn$ and $b(\kappa):=\sum b_j\kappa_j\in H^*(\overline{\modm}_{g,n},\br)$
define
\[
Z^\Omega_{b(\kappa)}(s,\hbar,t_0,t_1,...):=\exp\sum_{g,n}\frac{\hbar^{g-1}}{n!}\sum_{\vec{k}\in\bn^n}\sum_{m=0}^\infty\frac{s^m}{m!}\int_{\overline{\modm}_{g,n+m}}\hspace{-5mm}\Omega_{g,n+m}^{(1^n,0^m)}.\prod_{i=1}^n\psi_i^{k_i}t_{k_i}e^{b(\kappa)}.
\]
The function $Z^\Omega_{b(\kappa)}$ can be expressed as the following translation of $Z^\Omega$
\begin{equation}   \label{kaptrans}
Z^\Omega_{b(\kappa)}(s,\hbar,t_0,t_1,...)=Z^\Omega(s,\hbar,t_0,t_1+p_1(b),...,t_j+p_j(b),....)
\end{equation}
where $p_j(b)$ is a weighted homogeneous polynomial in $b$ of degree $j$ defined by
\[ 1-\exp\Big(\hspace{-1mm}-\hspace{-1mm}\sum_{j=1}^\infty b_iz^i\Big)=\sum_{j=1}^\infty p_j(b_1,...,b_j)z^j.\]

This is a shifted version of the result of Manin and Zograf \cite{MZoInv}.  The proof of \eqref{kaptrans} in the $s=0$ case---see \cite[Theorem 5.7]{NorEnu}---relies on the property 
$\Omega_{g,n+1}^{(1^{n+1})}=\psi_{n+1}\pi^*\Omega_{g,n}^{(1^n)}$.  The proof immediately generalises due to the same property 
\[\Omega_{g,n+1}^{(\sigma,1)}=\psi_{n+1}\pi^*\Omega_{g,n}^\sigma.\]   
This produces the pushforward relations
\[\pi_*(\Omega_{g,n+1}^{(\sigma,1)}\psi_{n+1}^{m})=\Omega_{g,n+1}^{(\sigma)}\kappa_m
\]
which produces \eqref{kaptrans} following the proof of \cite[Theorem 5.7]{NorEnu}.

For the super Weil-Petersson volumes, apply \eqref{kaptrans} when $b(\kappa)=2\pi^2\kappa_1$ to get
\begin{align*}\label{Zok}
Z^\Omega_{2\pi^2\kappa_1}(s,\hbar,t_0,t_1,...)&=\exp\sum_{g,n}\frac{\hbar^{g-1}}{n!}\sum_{\vec{k}\in\bn^n}\sum_{m=0}^\infty\frac{s^m}{m!}\int_{\overline{\modm}_{g,n+m}}\hspace{-5mm}\Omega_{g,n+m}^{(1^n,0^m)}.\prod_{i=1}^n\psi_i^{k_i}t_{k_i}e^{2\pi^2\kappa_1}\\
&=Z^\Omega(s,\hbar,t_0,t_1+2\pi^2,...,t_j-(-2\pi^2)^j/j!,....).
\end{align*}

By Theorem~\ref{BGW=spin} and \eqref{GBGWV}, $Z^\Omega$ satisfies Virasoro constraints.  Thus the partition function $Z^\Omega_{2\pi^2\kappa_1}(s,\hbar,t_0,t_1,...)$ satisfies the conjugation of these constraints by the specialisation of the translation \eqref{kaptrans}, and this is proven in \cite{NorEnu} to give the Stanford-Witten recursion \eqref{reconj}.  This proves the proposition.
\end{proof}

By \eqref{GBGWV}, the Virasoro constraints \eqref{virspin} do indeed hold, so an immediate corollary is a proof of Theorem 2.

\subsection{Spin class intersection calculations} \label{calc}
In this section we calculate all of the genus 0 spin class intersection numbers defined in \eqref{spincor}.  The TRR \eqref{TRRspin} uniquely determines the genus 0 $n$-point correlators, for $n\geq 3$ from 2-point correlators.  The 1-point and 2-point genus 0 correlators can be calculated via Theorem~\ref{BGW=spin} but instead we prove them directly here via algebraic geometry.  

The genus 0 part of $Z^\Omega(s,\hbar,t_0,t_1,...)$ is given by the factor $\exp(\hbar^{-1}F_0)$ and $F_0$ is decomposed into degree $n$ parts, denoted by $F_0=\sum_n\frac{1}{n!}F_{0,n}$ for
\[F_{0,n}(s,t_0,t_1,...):=\sum_{\vec{k}\in\bn^n}\langle \prod_{j=1}^n\tau_{k_j}\rangle_0\prod_{j=1}^nt_{k_j}\]
for
\[
\langle \prod_{j=1}^n\tau_{k_j}\rangle_0=\frac{1}{m!}\int_{\overline{\modm}_{0,n+m}}\hspace{-4mm}\Omega_{0,n+m}^{(1^n,0^m)}\prod_{j=1}^n\psi_j^{k_j}.
\]
\begin{lemma}
\begin{equation}  \label{1point}
 F_{0,1}=\sum_m\frac{2^{-(m+1)}}{(m+1)(2m+1)}\frac{t_{m}}{m!}s^{2m+2}
\end{equation}
\end{lemma}
\begin{proof}
Use the genus 0 relation for $m>0$:
\[ \psi_1^m=\psi_1^{m-1}\sum_{\underset{2,3\in I^c}{1\in I}}D_I
\]
where all marked points except for $p_1$ are Ramond points.  Then
\begin{align*}
\langle\tau_{m}\rangle_0=\frac{1}{m!}\int_{\overline{\modm}_{0,2m+3}}\hspace{-4mm}\psi_1^m\Omega_{0,2m+3}^{(1,2m+2)}
=\frac{1}{m!}\sum_{\underset{2,3\in I^c}{1\in I}}\int_{\overline{\modm}_{0,2m+3}}\hspace{-5mm}\psi_1^{m-1}D_I\Omega_{0,2m+3}^{(1,2m+2)}.
\end{align*}
The intersection number in each summand is given by
\begin{align*}
\int_{\overline{\modm}_{0,2m+1}}\hspace{-5mm}\psi_1^{m-1}D_I\Omega_{0,2m+3}^{(1,2m+2)}
=\int_{\overline{\modm}_{0,|I|+1}}\hspace{-5mm}\psi_1^{m-1}\Omega_{0,|I|+1}^{2,|I|-1}\cdot\int_{\overline{\modm}_{0,|I^c|+1}}\hspace{-5mm}\Omega_{0,|I^c|+1}^{1, |I^c|}.
\end{align*}
Now for $\sigma\in\{0,1\}^n$, 
\[\deg\Omega_{0,n}(e_\sigma)=\frac12(n+|\sigma|)-2= n-3\qquad\Leftrightarrow \qquad|\sigma|=n-2\] 
hence the second factor in the summand vanishes unless $|I^c|=2$.  So $I^c=\{2,3\}$ and only one summand is non-zero, leaving
%\begin{align*}\int_{\overline{\modm}_{0,2m+3}}\hspace{-2mm}\psi_1^m\Omega_{0,2m+3}&(e_1\otimes e_0^{\otimes 2m+2})\\&=\int_{\overline{\modm}_{0,2m+2}}\hspace{-5mm}\psi_1^{m-1}\Omega_{0,2m+2}(e_1^{\otimes 2}\otimes e_0^{\otimes 2m})\int_{\overline{\modm}_{0,3}}\hspace{-5mm}\Omega_{0,|I^c|+1}(e_1\otimes e_0^{\otimes 2})\\&=\int_{\overline{\modm}_{0,2m+2}}\hspace{-5mm}\psi_1^{m-1}\Omega_{0,2m+2}(e_1^{\otimes 2}\otimes e_0^{\otimes 2m})\\&=\int_{\overline{\modm}_{0,2m+2}}\hspace{-5mm}\psi_1^{m-1}\psi_2\pi^*\Omega_{0,2m+1}(e_1\otimes e_0^{\otimes 2m})\\&=\int_{\overline{\modm}_{0,2m+2}}\hspace{-5mm}\psi_2\pi^*\big(\psi_1^{m-1}\Omega_{0,2m+1}(e_1\otimes e_0^{\otimes 2m})\big)\\&=(2m-1)\int_{\overline{\modm}_{0,2m+1}}\hspace{-5mm}\psi_1^{m-1}\Omega_{0,2m+1}(e_1\otimes e_0^{\otimes 2m})\\&=(2m-1)!!\end{align*}
\[\int_{\overline{\modm}_{0,2m+3}}\hspace{-4mm}\psi_1^m\Omega_{0,2m+3}^{(1,2m+2)}=\int_{\overline{\modm}_{0,2m+1}}\hspace{-4mm}\psi_1^{m-1}\Omega_{0,2m+1}^{(2,2m-1)}\cdot\int_{\overline{\modm}_{0,3}}\hspace{-4mm}\Omega_{0,3}^{(1,2)}=(2m-1)!!
\]
%\[\langle\tau_{m}\rangle_0=\frac{\langle\tau_{m-1}\tau_0\rangle_0\langle\tau_0\rangle_0}{(2m+1)(2m+2)}=\frac{\langle\tau_{m-1}\tau_0\rangle_0}{(2m+1)(2m+2)}=\frac{(2m-1)\langle\tau_{m-1}\rangle_0}{(2m+1)(2m+2)}\]
which uses $\langle\tau_0\rangle_0=1$, the dilaton equation \eqref{dilaton} and induction.  %The denominator of $(2m+1)(2m+2)$ is due to the factors $\frac{1}{(2m+2)!}$ and $\frac{1}{(2m)!}$ which appear in the correlators \eqref{spincor}.
Hence
\[ F_{0,1}=\sum_m\frac{s^{2m+2}}{(2m+2)!}(2m-1)!!t_m=\sum_m\frac{2^{-(m+1)}}{(m+1)(2m+1)}\frac{t_{m}}{m!}s^{2m+2}\]
as required.

Note that although the TRR \eqref{TRRspin} involves $n$-point correlators for $n\geq 3$, \eqref{TRRspin} can be indirectly used to calculate 1-point correlators, once one relates the 1-point correlators to 3-point correlators via the dilaton equation \eqref{dilaton}, or equivalently the first of the Virasoro constraints given in \eqref{GBGWV}.  The proof above gave a direct algebro-geometric proof of \eqref{1point} which is more fundamental than the proof of \eqref{TRRspin}, instead of the indirect algebraic route.

\end{proof}

\begin{lemma}
\begin{equation}  \label{2point}
F_{0,2}=\sum_{m_1,m_2}\frac{2^{-(m_1+m_2+1)}}{m_1+m_2+1}\frac{t_{m_1}t_{m_2}}{m_1!m_2!}s^{2m_1+2m_2+2}\end{equation}
\end{lemma}
\begin{proof}
\[\langle\tau_{m_1}\tau_{m_2}\rangle_0=\frac{1}{(2|m|+2)!}\int_{\overline{\modm}_{0,2m+4}}\hspace{-2mm}\psi_1^{m_1}\psi_2^{m_2}\Omega_{0,2|m|+4}^{(2,2|m|+2)}
\]
Via $\displaystyle\psi_1^m=\psi_1^{m-1}\sum_{\underset{2,3\in I^c}{1\in I}}D_I$ for $m>0$, the non-zero summands consist of those terms with $2m_1$ Ramond points in the $I$ component.  Hence
\begin{align*}
\int_{\overline{\modm}_{0,2m+4}}\hspace{-8mm}\psi_1^{m_1}\psi_2^{m_2}\Omega_{0,2|m|+4}^{(2,2|m|+2)}&=\binom{2|m|+1}{2m_1}\int_{\overline{\modm}_{0,2m_1+2}}\hspace{-8mm}\psi_1^{m_1-1}\Omega_{0,2m_1+2}^{(2,2m_1)}\cdot\int_{\overline{\modm}_{0,2m_2+4}}\hspace{-8mm}\psi_1^{m_2}\Omega_{0,2m_2+4}^{(2,2m_2+2)}
\end{align*}
so
\begin{align*}
\langle\tau_{m_1}\tau_{m_2}\rangle_0&=\frac{2m_2+2}{2|m|+2}\langle\tau_{m_1-1}\tau_0\rangle_0\langle\tau_0\tau_{m_2}\rangle_0=\frac{2m_2+2}{2|m|+2}\frac{1}{2^{m_1}m_1!}\frac{1}{2^{m_2+1}(m_2+1)!}\\
&=\frac{1}{2|m|+2}\frac{1}{2^{m_1}m_1!}\frac{1}{2^{m_2}m_2!}
\end{align*}
which is symmetric in $m_1$ and $m_2$.
This gives
\begin{align*}
F_{0,2}&=\sum_{m_1,m_2}\frac{2^{-(m_1+m_2+1)}}{m_1+m_2+1}\frac{t_{m_1}t_{m_2}}{m_1!m_2!}s^{2m_1+2m_2+2}.
\end{align*}

\end{proof}

For $I\subset\{1,...,n\}$, define $m_I=\{m_i\}_{i\in I}$.
\begin{proposition}
\[F_0=\sum_n\frac{1}{n!}F_{0,n}=\sum_n\sum_{\vec{m}\in\bn^n}2^{-|m|}\frac{(2|m|+n-1)!}{(2|m|+2)!n!}\prod\frac{t_{m_i}}{m_i!}s^{2|m|+2}\]
\end{proposition}
\begin{proof}
The TRR
\[\langle \prod_{j=1}^n\tau_{k_j}\rangle_0=\hspace{-4mm}\sum_{I\subset\{4,...,n\}}\hspace{-4mm}\langle \tau_{0}\tau_{k_1-1}\prod_{j\in I}\tau_{k_j}\rangle_0
\langle\tau_{0} \tau_{k_2} \tau_{k_3}\prod_{j\in I^c}\tau_{k_j}\rangle_0\] 
implies that
\[ F_{0,n}=\sum_{\vec{m}}2^{-|m|-1}P(m_1,...,m_n)\prod\frac{t_{m_i}}{m_i!}s^{2|m|+2}
\]
where $P(m_1,...,m_n)$ is a symmetric degree $n-3$ polynomial defined recursively by
\[P(m_1,...,m_n)=m_1\hspace{-4mm}\sum_{I\subset\{4,...,n\}}\hspace{-4mm}P(m_1-1,m_I,0)P(m_2,m_3,m_{I^c},0).
\]
The initial conditions of the recursion are given by the special cases 
\[P(m_1)=\frac{1}{(m_1+1)(2m_1+1)},\qquad P(m_1,m_2)=\frac{1}{m_1+m_2+1}.\]

Instead of using the recursive definition of the polynomial $P$ directly, we use only the fact that it shows that $P$ is a symmetric degree $n-3$ polynomial.  Such a polynomial is uniquely determined by evaluation of one variable.  Hence it is uniquely determined by $P(m_1,m_2,m_3)\equiv 2$ together with
\begin{equation}  \label{eval0}
P(0,m_1,...,m_n)=P(m_1,...,m_n)\sum_{i=1}^n(2m_i+1)
\end{equation}
which is a consequence of the relation \eqref{dilaton}.  The symmetric polynomial given by 
\[ P(m_1,...,m_n)=2(2|m|+n-1)(2|m|+n-2)...(2|m|+3)=\frac{2(2|m|+n-1)!}{(2|m|+2)!}
\] 
satisfies \eqref{eval0} and the initial condition, which proves the claim.
\end{proof}

This proposition, combined with Theorem \ref{BGW=spin}, provides a new, geometric proof of the expression of the genus 0 part of the generalised Br\'ezin-Gross-Witten tau function conjectured in \cite[Conjecture 3.8]{AleCut}, for another proof see
\cite{YZ}.

%\[F_{0,3}=\sum_{\vec{m}}2^{-|m|}\frac{t_{m_1}t_{m_2}t_{m_3}}{m_1!m_2!m_3!}s^{2|m|+2}\]

\section{Spectral curve} \label{TR}

Many of the constructions in this paper can be formulated using topological recursion which gives an efficient way to package the partition functions and relations between them.

Topological recursion constructs a collection of correlators $\omega_{g,n}(p_1, \ldots, p_n)$, for $p_i\in C$ starting from a spectral curve $(C,B,x,y)$, which consists of a compact Riemann surface $C$, a symmetric bidifferential $B$ defined on $C\times C$, and meromorphic functions $x,y:C\to\bc$.   The zeros of $dx$ are required to be simple and disjoint from the zeros of $dy$.   This framework, developed by Chekhov, Eynard, and Orantin \cite{CEyHer,EOrInv}, originated from the loop equations satisfied by matrix models.  

Pictorially the recursion can be represented by the diagram below, which represents the sum  over all decompositions of a surface into a pair of pants and simpler surfaces given precisely in \eqref{EOrec}.
For $2g-2+n>0$ and $L = \{2, \ldots, n\}$, define
\begin{align}  \label{EOrec}
\omega_{g,n}(p_1,p_L)=&\sum_{\alpha}\Res_{p=\alpha}K(p_1,p) \bigg[\omega_{g-1,n+1}(p,\hat{p},p_L)\\
&\hspace{3cm}+ \mathop{\sum_{g_1+g_2=g}}_{I\sqcup J=L}^\circ \omega_{g_1,|I|+1}(p,p_I) \, \omega_{g_2,|J|+1}(\hat{p},p_J) \bigg]
\nonumber
\end{align}
where the outer summation is over the zeros $\alpha$ of $dx$.
\begin{center}
\includegraphics[scale=0.4]{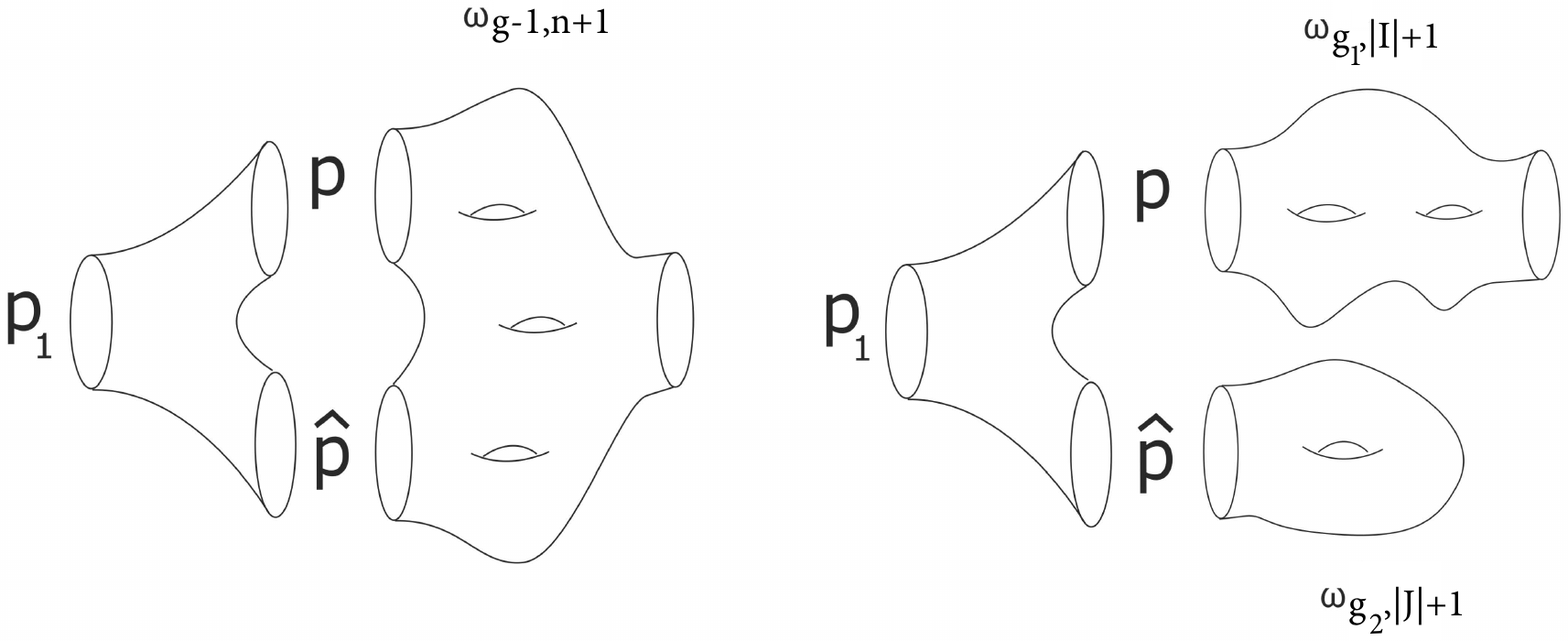}
\end{center}
We exclude terms that involve $\omega_1^0$ in the inner summation, which is denoted by an upper $\circ$.  Since the zeros of $dx$ are simple there is a local involution near any zero $\alpha$, denoted $p\mapsto\hat{p}$ which is the unique point $\hat{p}\neq p$ close to $\alpha$ satisfying $x(\hat{p})=x(p)$.  The kernel $K$ is defined by 
\[
K(p_1,p)=\frac{-\int^p_{\hat{p}}\omega_{0,2}(p_1,p')}{2[y(p)-y(\hat{p})] \, dx(p)}
\] 
which is a locally defined meromorphic function in a neighbourhood of each zero $\alpha$ of $dx$, and
\[\omega_{0,2}=B(p_1,p_2).\]
The poles of the correlator $\omega_{g,n}$ occur precisely at the zeros of $dx$.  When $C$ is rational, as will be the case in this paper, $B= \frac{d z_1 d z_2}{(z_1-z_2)^2}$, for $z_j=z(p_j)$ with respect to a rational coordinate $z$ for $C$. 

There are alternative definitions of topological recursion, equivalent to the definiiton above under certain conditions.  Among these is generalised topological recursion, which allows more general spectral curves, recently defined in \cite{ABDKSDeg}.

A fundamental example of topological recursion is given by the Airy curve
\[
C_{\text{Airy}}=\left(\bc,x=\frac{1}{2}z^2,\ y=z,\ B=\frac{dzdz'}{(z-z')^2}\right).
\]
The topological recursion correlators of $S_{\text{Airy}}$ store intersection numbers \cite{EOrTop} via
\[\omega_{g,n}^{\text{Airy}}=\sum_{\vec{k}\in\bn^n}\int_{\overline{\modm}_{g,n}}\prod_{i=1}^n\psi_i^{k_i}(2k_i+1)!!\frac{dz_i}{z_i^{2k_i+2}}.
\]
The $n=1$ correlators $\omega_{g,1}^{\text{Airy}}(z)$ are given by linear combinations of the differentials $\xi_k(z)=(2k+1)!!z^{-(2k+2)}dz$, and more generally for $2g-2+n>0$ each $\omega_{g,n}^{\text{Airy}}(z_1,...,z_n)$ is polynomial in the $\xi_k(z_i)$.  These polynomials produce the Kontsevich-Witten tau function via:
\[
Z^{\text{KW}}(\hbar,t_0,t_1,...)=\left.\exp\sum_{g,n}\frac{\hbar^{g-1}}{n!}\omega_{g,n}^{\text{Airy}}\right|_{\xi_k(z_i)=t_k}.\]
A similar story holds for the BGW tau function.  The Bessel curve \cite{DNoTop} is defined by:
\[
C_{\text{Bes}}=\left(\bc,x=\frac{1}{2}z^2,\ y=\frac{1}{z},\ B=\frac{dzdz'}{(z-z')^2}\right)
\]
and its correlators store intersection numbers \cite{CGGRel}
\[\omega_{g,n}^{\text{Bes}}=\sum_{\vec{k}\in\bn^n}\int_{\overline{\modm}_{g,n}}\Theta_{g,n}\prod_{i=1}^n\psi_i^{k_i}(2k_i+1)!!\frac{dz_i}{z_i^{2k_i+2}}
\]
which assemble to produce
\[Z^{\text{BGW}}(s=0,\hbar,t_0,t_1,...)=\left.\exp\sum_{g,n}\frac{\hbar^{g-1}}{n!}\omega_{g,n}^{\text{Bes}}\right|_{\xi_k(z_i)=t_k}.
\]

The translation of the Kontsevich-Witten tau function given in \eqref{KWtrans}, corresponds to a well-understood change \cite{CNoTop,DOSSIde} in $C_{\text{Airy}}$ which produces the following spectral curve:
\[
C_{\bk}=\left(\bc,x=\frac{1}{2}z^2,\ y=\frac{z}{z^2+s^2},\ B=\frac{dzdz'}{(z-z')^2}\right)
\]
with correlators
\[\omega_{g,n}^{\bk}=s^{4g-4+2n}\sum_{\vec{k}\in\bn^n}\int_{\overline{\modm}_{g,n}}\hspace{-2mm}\bk(s^{-2})\prod_{i=1}^n\psi_i^{k_i}(2k_i+1)!!\frac{dz_i}{z_i^{2k_i+2}}
\]
that satisfy $\displaystyle\lim_{s\to 0} C_{\bk}=C_{\text{Airy}}$ and $\displaystyle\lim_{s\to 0}\omega_{g,n}^{\bk}=\omega_{g,n}^{\text{Airy}}$.  (Note that equality of the limit of the correlators of the spectral curve with the correlators of the limit of the spectral curve is non-trivial.)

The spectral curve $C_{\bk}$ is also used to produce spin intersection numbers. They correspond to an expansion of the same correlators $\omega_{g,n}^{\bk}$ at $z=\infty$, but in a local coordinate $\eta^{-1}=(z^2+s^2)^{-\frac{1}{2}}$, namely for $2g-2+n>0$
\begin{equation}\label{omOmm}
\omega_{g,n}^{\bk}=\sum_{\vec{k}\in\bn^n}\sum_{m=0}^\infty\frac{s^m}{m!}\int_{\overline{\modm}_{g,n+m}}\hspace{-5mm}\Omega_{g,n+m}^{(1^n,0^m)}.\prod_{i=1}^n\psi_i^{k_i}(2k_i+1)!!\frac{d\eta_i}{\eta_i^{2k_i+2}}.
\end{equation}
This relation immediately follows from \eqref{KeqOm} and the well-known interpretation of the Virasoro group action in terms of the deformation of the spectral parameter. 

Equivalently, one can consider the spectral curve $C_{\bk}$ with a shifted $x$, namely $x_\Omega=x+\frac{s^2}{2}$. It satisfies the spectral curve equation 
$4x_\Omega^2y^2-2x_\Omega+s^2=0$, that was conjectured in \cite{AleCut} to be the spectral curve for topological recursion for the generalized BGW tau-function with the expansion at $x_\Omega$ in the local coordinate $x_\Omega^{-1/2}$.  %That is, 
%\[
%C_{\Omega}=\left(\bc,x=\frac{1}{2}z^2,\ y=\frac{z}{z^2+s^2},\ B=\frac{dzdz'}{(z-z')^2}\right)
%\]

%using the notation in \eqref{spincor}, via
%\[\Res_{x_1=\infty}...\Res_{x_n=\infty}\prod_{i=1}^n(x_i+\tfrac12s^2)^{k_i+\frac12}\omega_{g,n}^{\bk}=\sum_{m\geq 0}s^m\sum_{\vec{k}\in\bn^n}\langle\prod_{i=1}^n\tau_{k_i}\rangle_g.
%\]

%Equivalently, as in \cite{AleCut} the spectral curve is given by $4x^2y^2-2x+s^2=0$, or $x=\frac12(z^2+s^2)$ and $y=\frac{z}{z^2+s^2}$ which agrees with $C_\bk$ up to a shift in $x$, which does not affect the correlators $\omega_{g,n}$ .
 
The super Weil-Petersson volumes of the moduli space of super curves with Neveu-Schwarz punctures is obtained by a well-understood change \cite{NorEnu} of the Bessel spectral curve to produce the spectral curve
\[ C_{NS}=\left(\bc,\ x=\frac12z^2,\ y=\frac{\cos(2\pi z)}{z},\ B=\frac{dzdz'}{(z-z')^2}\right).\]
Topological recursion applied to $C_{NS}$ generates correlators
\[\omega_{g,n}=\frac{\partial}{\partial z_1}...\frac{\partial}{\partial z_n}\cl_z\Big\{\widehat{V}_{g,n}\WP(s=0,L_1,...,L_n)\Big\}dz_1...dz_n\]
where
\[\cl_z\{P(L_1,...,L_n)\}=\int_0^\infty...\int_0^\infty P(L_1,...,L_n)\prod_{i=1}^n \exp(-z_iL_i)dL_i\] denotes the Laplace transform.
%For general $s$, we do not yet know a spectral curve that stores $\widehat{V}_{g,n}\WP(s,L_1,...,L_n)$.
    For general $s$, we expect that a spectral curve for $\widehat{V}_{g,n}\WP$ can be obtained by deformation of the spectral curve $C_{NS}$.  We do not yet know of an expression for such a spectral curve.

\end{document}